\documentclass[12pt,reqno]{amsart}

\usepackage{amsmath, amsthm, amssymb}
\usepackage{amsfonts}
\usepackage[ansinew]{inputenc}
\usepackage[dvips]{epsfig}
\usepackage{graphicx}
\usepackage[english]{babel}
\usepackage{graphicx}
\usepackage{hyperref}
\usepackage{bbm} 
\usepackage{cite}
\usepackage{graphicx}
\usepackage{amscd}
\usepackage{xcolor}
\usepackage{bm}
\usepackage{enumerate}

\usepackage{verbatim}
\usepackage{hyperref}
\usepackage{amstext}
\usepackage{latexsym}
\usepackage{marginnote}
\let\oldsqrt\sqrt
\def\sqrt{\mathpalette\DHLhksqrt}
\def\DHLhksqrt#1#2{%
\setbox0=\hbox{$#1\oldsqrt{#2\,}$}\dimen0=\ht0
\advance\dimen0-0.2\ht0
\setbox2=\hbox{\vrule height\ht0 depth -\dimen0}%
{\box0\lower0.4pt\box2}}

\allowdisplaybreaks
\newcommand{\R}{\mathbb{R}} 
\newcommand{\N}{\mathbb{N}} 

\newcommand{\dist}{\textnormal{dist}} 
\newcommand{\diam}{\textnormal{diam}} 
\newcommand{\supp}{\textnormal{supp}} 

\newcommand{\ov}{\overline}

\renewcommand{\phi}{\varphi}

\theoremstyle{definition}
\newtheorem{defi}{Definition}[section]
\newtheorem{remark}[defi]{Remark}

\theoremstyle{plain} 
\newtheorem{thm}[defi]{Theorem}
\newtheorem{prop}[defi]{Proposition}
\newtheorem{lemma}[defi]{Lemma}

\newtheorem{cor}[defi]{Corollary}

\theoremstyle{definition}

\numberwithin{equation}{section}

\title{Fractional Hardy inequalities on $C^{1,1}$ open sets}
\author[Abdelrazek Dieb and Remi Yvant Temgoua]{Abdelrazek Dieb$^{1}$ and Remi Yvant Temgoua$^{2}$} 

\address{$^1$ Department of Mathematics, Faculty of Mathematics and Computer Science \\
University Ibn Khaldoun of Tiaret, Algeria\\
and\\
Laboratoire d'Analyse Nonlin\'eaire et Math\'ematiques Appliqu\'ees\\
Abou Bakr Belkaid University, Tlemcen, Algeria 
}
\email{abdelrazek.dieb@univ-tiaret.dz}

\address{$^2$ The Fields Institute for research in mathematical sciences, 222 College Street, 2nd floor, Toronto, Ontario, m5t 3j1 Canada, and School of Mathematics and Statistics, Carleton University, 1125 Colonel By Dr, Ottawa, Ontario, K1S 5B6, Canada, and University of Bertoua, P.O Box: 416 Bertoua, Cameroon} 
\email{rtemgoua@fieldsinstitute.ca, remi.y.temgoua@aims-senegal.org} 


\begin{document}

\begin{abstract}
Let $\Omega$ be a bounded open set of class $C^{1,1}$ in $\R^N$ and $s\in(\frac{1}{2}, 1)$. We study a family of fractional Hardy-type inequalities 
\begin{equation}\label{a0}
\frac{c_{N,s}}{2}\displaystyle\iint_{\Omega\times\Omega}\frac{(u(x)-u(y))^2}{|x-y|^{N+2s}}\ dxdy-\displaystyle\lambda\int_{\Omega}u^2\ dx\geq C\displaystyle\int_{\Omega}\frac{u^2}{\delta^{2s}}\ dx,~~~\quad\forall\lambda\in\R,
\end{equation}
with $u\in C_c^\infty(\Omega)$ and $C=C(\Omega,s,N,\lambda)>0$. We show that the best constant in \eqref{a0} is achieved if and only if $\lambda>\lambda^*(s,\Omega)$, for some $\lambda^*(s,\Omega)\in\R$. As a by-product, we derive in particular that the best constant in Hardy inequality $\mu_{N,s}(\Omega)$ is achieved if and only if $\mu_{N,s}(\Omega)<\mathfrak{h}_{N,s}$, with $\mathfrak{h}_{N,s}$ being the best constant for the fractional Hardy inequality in the half space. Moreover, if $\Omega$ is a convex open set, we obtain a lower bound for $\lambda^*(s,\Omega)$ in terms of the volume of $\Omega$. Specifically, we prove that $\lambda^*(s,\Omega)\geq a(N,s)|\Omega|^{-\frac{2s}{N}}$ with an explicit constant $a(N,s)>0$. Finally, for bounded $C^{1,1}$ domains, we prove that, for s sufficiently close to $\frac 12$, the optimal Hardy constant is independent of both the geometry and the topology of $\Omega$. More precisely, we establish that $\mu_{N,s}(\Omega)=\mathfrak{h}_{N,s}.$
This behavior is in sharp contrast with the local case, where the topology/geometry of the domain strongly influences the value of the optimal constant, and reveals a new rigidity phenomenon in the nonlocal setting. 
\end{abstract}

\maketitle 

\section{Introduction and main results}\label{section:introduction}

Let $\Omega$ be a bounded open set in $\R^N$ and $s\in\left(\frac{1}{2},1\right)$. The fractional Hardy inequality in $\Omega$ states that (see e.g., \cite{chen2003hardy,dyda2004fractional})
\begin{equation}\label{hardy-inequality-in-domains}
\frac{c_{N,s}}{2}\iint_{\Omega\times\Omega}\frac{(u(x)-u(y))^2}{|x-y|^{N+2s}}\ dxdy\geq c\int_{\Omega}\frac{u^2}{\delta^{2s}}\ dx,~~~~~\forall u\in C^\infty_c(\Omega),
\end{equation}
where $c=c(N,s,\Omega)$ is a positive constant that depends only on $N, s$
, and $\Omega$, $\delta(x)=\delta_{\Omega}(x)=\dist(x,\partial\Omega)$, and $c_{N,s}$ a normalization constant such that
\begin{equation}\label{nonlocal-to-local}
    (-\Delta)^s_{\Omega}u\to-\Delta u~~~\textnormal{as}~~~~s\to1^-,
\end{equation} 
which is given explicitly by $c_{N,s}=2^{2s}\pi^{-\frac{N}{2}}s\frac{\Gamma(\frac{N+2s}{2})}{\Gamma(1-s)}$. Here, $(-\Delta)^s_{\Omega}$ denotes the regional fractional Laplacian which acts on smooth enough functions as 
\begin{equation*}
(-\Delta)^s_{\Omega}u(x)=c_{N,s}P.V.\int_{\Omega}\frac{u(x)-u(y)}{|x-y|^{N+2s}}\ dy=c_{N,s}\lim_{\varepsilon\to0^+}\int_{\{|x-y|\geq\varepsilon\}}\frac{u(x)-u(y)}{|x-y|^{N+2s}}\ dy.
\end{equation*}
This operator is the infinitesimal generator of the so-called \textit{censored stable L\'{e}vy processes} and has received great attention in recent years, see e.g. \cite{bogdan2003censored,chen2003hardy,dyda2004fractional,fall2023existence} and the references therein. The censored stable process is a jump process with jumps restricted to the underlying open set $\Omega$. Specifically speaking, jumps between $\Omega$ and its complement are suppressed. 

If $s=1$, then \eqref{nonlocal-to-local} tells that $(-\Delta)^s_{\Omega}$ coincides with the classical Laplacian $-\Delta$. In \cite{brezis1997hardy}, Brezis and Marcus studied a family of Hardy-type inequalities in $L^2$
\begin{equation}\label{local-minimization-problem}
J_{\lambda}(\Omega)=\inf_{u\in H^1_0(\Omega)}\frac{\displaystyle\int_{\Omega}|\nabla u|^2\ dx-\lambda\int_{\Omega}u^2\ dx}{\displaystyle\int_{\Omega}\frac{u^2}{\delta^2}\ dx},~~~\forall \lambda\in\R,
\end{equation}
and established existence and non-existence results. Precisely, among other results, they proved that for every bounded domain of class $C^2$, the infimum \eqref{local-minimization-problem} is achieved if and only if $J_{\lambda}(\Omega)<\frac{1}{4}$ with $\lambda>\lambda^*$ for some $\lambda^*\in\R$. Note that $\frac{1}{4}$ is the best constant in classical Hardy inequality on the half-space, namely, 
\begin{equation}\label{hardy-half-space-local}
\frac{1}{4}\int_{\R^N_+}\frac{u^2}{x_N^2}\ dx\leq\int_{\R^N_+}|\nabla u|^2\ dx,~~u\in H^1_0(\R^N_+).
\end{equation}
The regularity assumption on the domain plays a key role in the arguments developed in \cite{brezis1997hardy}.  Indeed, many arguments used in this paper are based on the well-known \textit{tubular neighbourhood theorem}, which allows to use tubular coordinates near the boundary of a domain of class $C^2$. Moreover, the assumption that $\Omega$ is of class $C^2$ is used to guarantee that the distance function $\delta$ is of class $C^2$ in a neighborhood of the boundary: this allows in turn to construct a suitable admissible test function for $J_{\lambda}(\Omega)$. 
\par
Let us point out that the case $\lambda=0$ is of particular interest since, in this case, \eqref{local-minimization-problem} is related to the classical Hardy inequality
\begin{equation*}
\mu(\Omega)\displaystyle\int_{\Omega}\frac{|u|^2}{\delta^2}\ dx\leq\displaystyle\int_{\Omega}|\nabla u|^2\ dx,
\end{equation*}
where $u\in H^{1}_0(\Omega)$. It was shown in \cite{Dav94,marcus1998best} that if the open set $\Omega$ is sufficiently regular, then the sharp constant $\mu(\Omega)$ depends on $\Omega$ and satisfies $\mu(\Omega)\leq \frac{1}{4}$. Moreover, for convex open sets $\mu(\Omega)=\frac{1}{4}$, though there are smooth bounded domains such that $\mu(\Omega)< \frac{1}{4}$. See for instance \cite{Dav94, Dav, marcus1998best, hoffmann2002geometrical, LaSo, Lewis,marcus2000eigenvalue}.

Brezis-Marcus's result was extended to $L^p$ setting $(p\neq2)$ in \cite{marcus2000eigenvalue} where the authors studied the quantity 
\begin{equation*}
    J_{\lambda,p}(\Omega)=\inf_{u\in W^{1,p}_0(\Omega)}\frac{\displaystyle\int_{\Omega}|\nabla u|^p\ dx-\lambda\int_{\Omega}\eta \frac{|u|^p}{\delta^p}\ dx}{\displaystyle\int_{\Omega}\frac{|u|^p}{\delta^p}\ dx},
\end{equation*}
for all $\lambda\in\R$, under the assumptions
\begin{equation*}
    \eta\in C(\ov\Omega),~~\eta>0~\textnormal{in}~\Omega,~~\eta=0~\textnormal{on}~~\partial\Omega.
\end{equation*}
To the best of the authors' knowledge, no such result is available in the literature for the fractional setting yet. We aim to bridge this gap in the present paper.\\
Before going further, let us first comment on the Hardy inequality \eqref{hardy-inequality-in-domains}, although it is difficult to provide a comprehensive account of the entire literature.
\par
We begin by recalling a fundamental result due to Dyda \cite{dyda2004fractional}, which asserts that the fractional Hardy inequality holds in a bounded Lipschitz domain if and only if $2s>1$\footnote{Actually, a more general fractional Hardy inequality has been proved in  \cite{dyda2004fractional}.}. This can be viewed as an analog of a result by Ne\v cas~\cite{Necas}, asserting that if $\Omega$ is a bounded Lipschitz domain and $1<p<\infty$, then there is $\mu(\Omega)>0$
such that the $p$-Hardy inequality
\begin{equation}\label{e.p-hardy}
\mu(\Omega)\int_{\Omega} \frac{\lvert u(x)\rvert^p}{\delta(x)^p}\,dx
\le \int_\Omega \lvert \nabla u(x)\rvert^p\,dx
\end{equation}
holds for all $u\in C^\infty_c(\Omega)$

In \cite{loss2010hardy}, Loss and Sloane found the best constant in \eqref{hardy-inequality-in-domains}, that is,
\begin{equation}\label{a3}
\mu_{N,s}(\Omega):=\inf_{u\in C^\infty_c(\Omega)}\frac{\frac{c_{N,s}}{2}\displaystyle\iint_{\Omega\times\Omega}\frac{(u(x)-u(y))^2}{|x-y|^{N+2s}}\ dxdy}{\displaystyle\int_{\Omega}\frac{u(x)^2}{\delta(x)^{2s}}\ dx}. 
\end{equation}
They showed that if $\Omega$ is \textit{convex}, then $\mu_{N,s}(\Omega)$ is independent of $\Omega$ and coincides with that on the half-space which was earlier found in \cite{bogdan2011best}. More precisely, they proved that 
\begin{equation}\label{a4}
\mu_{N,s}(\Omega)=\mu_{N,s}(\R^N_+):=c_{N,s}\kappa_{N,2s},
\end{equation}
where $\kappa_{N,2s}$ is the constant introduced in \cite{loss2010hardy} given by
\begin{align}
\kappa_{N,2s}:=\pi^{\frac{N-1}{2}}\frac{\Gamma(\frac{1+2s}{2})}{\Gamma(\frac{N+2s}{2})}\frac{1}{2s}\left[\frac{2^{1-2s}}{\sqrt{\pi}}\Gamma(1-s)\Gamma\left(\frac{1+2s}{2}\right)-1\right].
\end{align}
For ease of notation, we use $\mathfrak{h}_{N,s}:=\mu_{N,s}(\R^N_+)$. Remark that 
\begin{equation*}
\mathfrak{h}_{N,s}\to\frac{1}{4}~~~\text{as}~~s\to 1^{-},
\end{equation*}
and thus, $\mathfrak{h}_{N,s}$ agrees, asymptotically as $s\to1^-$, with the best constant in \eqref{hardy-half-space-local}.
\par
The question of determining the sharp constant in \eqref{hardy-inequality-in-domains} has been addressed first in \cite{bogdan2011best} for the half-space and then in \cite{loss2010hardy} for convex sets. See also \cite{dyda2012fractional,frank2009sharp, Frankgs} for some extended results to the $L^p$ setting.
We stress that the problem of finding the sharp constant in \eqref{hardy-inequality-in-domains} is widely open for a general open set. Our results shed light on this problem, see the discussion below.
\par
We highlight the recent works \cite{dykija, DV, EUSV, Ihna}, in which the authors establish sharp necessary and sufficient conditions on $\Omega$ and $s$ ensuring the validity of the fractional Hardy inequality, that is, conditions under which the constant in \eqref{a3} is positive. In addition, we quote the recent work by Adimurthi et al. \cite{adi_roy} who established a generalization of Dyda's result, \cite{dyda2004fractional}, and, among other things, proved fractional Hardy inequalities for the remaining critical cases. Moreover, the critical one-dimensional case has also been considered in \cite{adi_jana}.

For completeness, we also mention another important class of fractional Hardy inequalities associated with the so-called restricted fractional Laplacian; see \cite{brasco-cinti, bianchi2024on, FMT} and the references therein.

\subsection{Main results.}

In the present paper, we study the following minimization problem
\begin{equation}\label{a1}
\mathfrak{J}_{\lambda,s}(\Omega)=\inf_{u\in H^s_0(\Omega)}\frac{\frac{c_{N,s}}{2}\displaystyle\iint_{\Omega\times\Omega}\frac{(u(x)-u(y))^2}{|x-y|^{N+2s}}\ dxdy-\lambda\int_{\Omega}u^2\ dx}{\displaystyle\int_{\Omega}\frac{u^2}{\delta^{2s}}\ dx},~~~\forall\lambda\in\R,
\end{equation}
where $\Omega$ is a bounded $C^{1,1}$-open set in $\R^N$. Here, $H^s_0(\Omega)$ is the completion of $C^{\infty}_c(\Omega)$ in the fractional Sobolev space $H^s(\Omega)$, where 
\begin{align}
H^s(\Omega)=\left\{ u \in L^2(\Omega): \quad \frac{u(x)-u(y)}{|x-y|^{\frac{N+2s}{2}}} \in L^2(\Omega\times\Omega)\right\}.
\end{align}
Notice that $H^s(\Omega)$ is a Hilbert space endowed with the natural norm
\begin{align}
\|\cdot\|^2_{H^s(\Omega)}=\|\cdot\|^2_{L^2(\Omega)}+[~\cdot~]_s,
\end{align}
where $[~\cdot~]_s$ stands for the Gagliardo semi-norm given by
\begin{align*}
[\phi]_s=\frac{c_{N,s}}{2}\iint_{\Omega\times\Omega}\frac{(\phi(x)-\phi(y))^2}{|x-y|^{N+2s}}\ dxdy.
\end{align*}
Moreover, for $\frac{1}{2}<s<1$ the space $H_0^s(\Omega)$ equipped with the norm
\begin{align*}
\|\phi\|^2_s:=[\phi]_s,
\end{align*}
is a Hilbert space.

Now, observe that the function 
$$\lambda\mapsto\mathfrak{J}_{\lambda,s}(\Omega),$$ 
is concave and non-increasing on $\R$, 
$$\mathfrak{J}_{\lambda,s}(\Omega)\to-\infty \,\text{ when } \lambda\to+\infty \quad\text{ and }\quad\; \mathfrak{J}_{\lambda,s}(\Omega)=0 \text{ for } \lambda=\lambda_{1,s},$$
where $\lambda_{1,s}$ is the first eigenvalue of $(-\Delta)^s_{\Omega}$ in $H^s_0(\Omega)$ with Dirichlet boundary conditions. Finally, we stress that, 
\begin{align*}
\mathfrak{J}_{\lambda,s}(\Omega)=\mu_{N,s}(\Omega)\quad \text{ for }\lambda=0.
\end{align*}
The first result of this paper is the following. 
\begin{thm}\label{first-main-result}
Let $s\in (\frac12,\,1)$. Then for every bounded open set $\Omega$ of class $C^{1,1}$ in $\R^N$, $N\geq 2$, there exists a constant $\lambda^*(s,\Omega)\in\R$ such that
\begin{align}
&\mathfrak{J}_{\lambda,s}(\Omega)=\mathfrak{h}_{N,s},~~~~\forall\lambda\leq\lambda^*(s,\Omega)\label{a6},\\
&\mathfrak{J}_{\lambda,s}(\Omega)<\mathfrak{h}_{N,s},~~~~\forall\lambda>\lambda^*(s,\Omega)\label{a7}.
\end{align}
Furthermore, the infimum in \eqref{a1} is achieved if and only if $\lambda>\lambda^*(s,\Omega)$. 
\end{thm}
In particular, we find that there exists $\lambda\in\R$ such that
\begin{equation}\label{a8}
\frac{c_{N,s}}{2}\iint_{\Omega\times\Omega}\frac{(u(x)-u(y))2}{|x-y|^{N+2s}}\ dxdy\geq\mathfrak{h}_{N,s}\int_{\Omega}\frac{u^2}{\delta^{2s}}\ dx+\lambda\int_{\Omega}u^2\ dx~~~~\forall u\in H^s_0(\Omega).
\end{equation}
Thus, the largest such constant is precisely $\lambda^*(s,\Omega)$, that is,
\begin{equation}\label{a9}
\lambda^*(s,\Omega)=\inf_{u\in H^s_0(\Omega)}\frac{\frac{c_{N,s}}{2}\displaystyle\iint_{\Omega\times\Omega}\frac{(u(x)-u(y))^2}{|x-y|^{N+2s}}\ dxdy-\mathfrak{h}_{N,s}\int_{\Omega}\frac{u^2}{\delta^{2s}}\ dx}{\displaystyle\int_{\Omega}u^2\ dx}.
\end{equation}
Owing to Theorem \ref{first-main-result}, the infimum in \eqref{a9} is not achieved. Notice that $\lambda^*(s,\Omega)$ is well-defined for every bounded Lipschitz open sets. In fact, for $u\in H^s_0(\Omega)$, the first term on the numerator in the above infimum is always finite. Moreover, by our assumption on the open set, we invoke \cite[Theorem 1.1]{dyda2004fractional} to see that the Hardy term $\displaystyle\int_{\Omega}\frac{u^2}{\delta^{2s}}\ dx$ is also finite.

Note that Theorem \ref{first-main-result} is the fractional analog of the main result of Brezis and Marcus \cite{brezis1997hardy}, which treats the local case. It is also important to note that if the infimum \eqref{a1} is achieved at a function $u$, then $u$ satisfies the Euler-Lagrange equation 
\begin{equation}\label{eigenvalue-problem}
(-\Delta)^s_{\Omega}u-\lambda u=\mathfrak{J}_{\lambda,s}(\Omega)\frac{u}{\delta^{2s}}~~~\text{in}~~\Omega. 
\end{equation}  
Thus, $\mathfrak{J}_{\lambda,s}(\Omega)$ can be considered as the principal weighted eigenvalue of $(-\Delta)^s_{\Omega}-\lambda$ with respect to the Hardy weight, and $u$ is a corresponding principal eigenfunction.\\

The proof of Theorem \ref{first-main-result} is divided into several  steps.\\

\textbf{Step 1.} First, we prove that $\mathfrak{J}_{\lambda,s}(\Omega)\leq\mathfrak{h}_{N,s}$ for all $\lambda\in \R$. The main idea uses the fact that $\mathfrak{h}_{N,s}$ is the sharp constant in the fractional Hardy inequality on the half-space, and then to consider a rescaled minimizing sequence that is supported in a cone with a vertex on the boundary of the open set where a geometric assumption is satisfied. \\

\textbf{Step 2.} We prove that $\mathfrak{J}_{\lambda,s}(\Omega)=\mathfrak{h}_{N,s}$ for some $\lambda$. The main ingredient in the proof of this result is the following inequality
\begin{equation}
\frac{c_{N,s}}{2}\iint_{\Omega_{\beta}\times\Omega_{\beta}}\frac{(u(x)-u(y))^2}{|x-y|^{N+2s}}\ dxdy\geq\mathfrak{h}_{N,s}\int_{\Omega_{\beta}}\frac{u(x)^2}{\delta(x)^{2s}}\ dx~~~~~~\forall u\in H^s_0(\Omega)
\end{equation}
(where $\Omega_{\beta}=\{x\in\Omega: \delta(x)<\beta\}$ is the so-called \textit{inner tubular neighborhood} of $\Omega$) which is valid for all $\beta$ sufficiently small.\\

\textbf{Step 3.} We prove that the infimum \eqref{a1} is achieved for every $\lambda>\lambda^*(s,\Omega)$. The proof follows essentially the same strategy as in \cite{brezis1997hardy}.\\

\textbf{Step 4.} We prove that the infimum \eqref{a1} is not achieved for any $\lambda\leq\lambda^*(s,\Omega)$. The proof requires a more delicate argument: it relies on the next non-existence result. 

\begin{thm}\label{non-existence-theorem}
Let $\Omega$ be a bounded open set in $\R^N$ of class $C^{1,1}$. Suppose that $u$ is a nonnegative function in $H^s_0(\Omega)\cap C(\Omega)$ satisfying the inequality 
	\begin{equation}\label{supersoution-lower-bound}
	(-\Delta)^s_{\Omega}u-\frac{\mathfrak{h}_{N,s}}{\delta^{2s}}u\geq-\frac{\eta}{\delta^{2s}}u~~~\text{in}~\Omega, 
	\end{equation}
	where $\eta$ is a continuous nonnegative function in $\overline{\Omega}$ such that
	\begin{equation}\label{condition-nonexistence}
	\lim\limits_{\delta(x)\to0}\eta(x)(\delta(x)^{1-s}\log \delta(x))^{2}=0.
	\end{equation}
	Then $u\equiv0$.
\end{thm}

As a byproduct of Theorem \ref{first-main-result}, we derive the following corollary.
\begin{cor}\label{first-main-result-corl}
The infimum in \eqref{a3} is achieved if and only if $$\mu_{N,s}(\Omega)<\mathfrak{h}_{N,s}.$$
Moreover, if $\mu_{N,s}(\Omega)$ admits a minimizer $u\in H^s_0(\Omega)$ then $u$ is, up to a constant, the unique weak solution to
\begin{align}
(-\Delta)^s_{\Omega}u=\mu_{N,s}(\Omega)~\frac{u}{\delta^{2s}}\quad \text{ in } \Omega,\quad u>0 \quad \text{ in } \Omega, \quad \text{ and }~~ u=0 \text{ on } \partial\Omega.
\end{align}
\end{cor}
The first part of the above corollary indicates that a minimizing sequence cannot concentrate at the boundary of the domain. This corollary presents some similarities with the study of fractional Sobolev inequalities
on domains, see \cite{fall2023existence,frank2025sharp,frank2018minimizers}. 

In \cite{brezis1997hardy} Brezis and Marcus proved that, if $\Omega$ is an arbitrary convex domain, the largest constant $\lambda(\Omega)$ in the inequality
\begin{equation*}
    \int_{\Omega}|\nabla u(x)|^2\ dx\geq\frac{1}{4}\int_{\Omega}\frac{u(x)^2}{\delta(x)^2}\ dx+\lambda(\Omega)\int_{\Omega}u(x)^2\ dx,~~~u\in H^1_0(\Omega),
\end{equation*}
satisfies
\begin{equation*}
    \lambda(\Omega)\geq\frac{1}{4\diam^2(\Omega)}.
\end{equation*}
They also questioned in \cite{brezis1997hardy} whether the diameter could be replaced by the volume of $\Omega$, i.e, whether $\lambda(\Omega)\geq\alpha|\Omega|^{-\frac{2}{N}}$ for some universal constant $\alpha>0$. This question was answered in the affirmative in \cite{hoffmann2002geometrical}, see also \cite{Tidblom}, where it was shown that $\lambda(\Omega)\geq c|\Omega|^{-\frac{2}{N}}$ with $c=c(N)=\displaystyle\frac{N^{\frac{N-2}{N}}|\mathbb{S}^{N-1}|^{\frac{2}{N}}}{4}$.

Our next result extends this result to fractional settings. It reads as follows. 

\begin{thm}\label{Geom-convex-Hardy}
Let $\Omega\subset\R^N$ be a convex open set. Then
\begin{equation*}\label{}
\frac{c_{N,s}}{2}\iint_{\Omega\times\Omega}\frac{(u(x)-u(y))^2}{|x-y|^{N+2s}}\ dxdy\geq\mathfrak{h}_{N,s}\int_{\Omega}\frac{u^2}{\delta^{2s}}\ dx+a(N,s)|\Omega|^{-\frac{2s}{N}}\int_{\Omega}u^2\ dx,
\end{equation*}
for all $u\in H^s_0(\Omega)$, where $a(N,s)$ is given by
\begin{equation}\label{geometrichardyconstant0}
a(N,s)=\mathfrak{h}_{N,s}2^{-2s}(2^{2s}-1)\frac{\sqrt{\pi}\,\Gamma(\frac{N+2s}{2})}{\Gamma(\frac{N}{2})\Gamma(\frac{1+2s}{2})}\left(\frac{N}{|\mathbb{S}^{N-1}|}\right)^{-\frac{2s}{N}}.
\end{equation}
\end{thm}
The proof of Theorem \ref{Geom-convex-Hardy} is a consequence of the following geometric improvement of the fractional Hardy inequality, see Section \ref{section:proof of second main result}.
\begin{thm}\label{Geom-Hardy} 
Let $\Omega\subset\R^N$ be an open set. Then
\begin{equation*}\label{}
\frac{c_{N,s}}{2}\iint_{\Omega\times\Omega}\frac{(u(x)-u(y))^2}{|x-y|^{N+2s}}\ dxdy\geq\mathfrak{h}_{N,s}\int_{\Omega}\frac{u^2}{m_{2s}^{2s}}\ dx+a(N,s)\int_{\Omega}\frac{u^2}{|\Omega_x|^{\frac{2s}{N}}}\ dx,
\end{equation*}
for all $u\in H^s_0(\Omega)$.
\end{thm}
Here $m_{2s}(x)$ denotes the fractional version of Davies's distance and $\Omega_x$ is the set of points $z$ in $\Omega$ such that the segment $\left[x,z\right]$ is contained in $\Omega$, see Section \ref{section:proof of second main result} for more details.

Furthermore, for bounded convex sets, one can improve the result in Theorem \ref{Geom-convex-Hardy}.

\begin{thm}\label{Geom-convex-Hardy-bndd}
Let $\Omega\subset\R^N$ be a bounded convex open set. Then
\begin{equation*}\label{}
\frac{c_{N,s}}{2}\iint_{\Omega\times\Omega}\frac{(u(x)-u(y))^2}{|x-y|^{N+2s}}\ dxdy\geq\mathfrak{h}_{N,s}\int_{\Omega}\frac{u^2}{\delta^{2s}}\ dx+\mathfrak{C}(\Omega,N,s)\int_{\Omega}u^2\ dx,
\end{equation*}
for all $u\in H^s_0(\Omega)$, where
\begin{equation}\label{geometrichardyconstant}
\mathfrak{C}(\Omega,N,s)=\left(a(N,s)+b(N,s)~\frac{|\Omega|^{\frac 1N}}{\diam(\Omega)}\right) |\Omega|^{-\frac{2s}{N}},
\end{equation}
and
\begin{equation}
b(N,s)=2^{2s-1}(4-2^{3-2s})\frac{\Gamma(\frac{N+2s}{2})}
{\Gamma(\frac{N}{2})\Gamma(1-s)}\left(\frac{N}{|\mathbb{S}^{N-1}|}\right)^{-\frac{2s-1}{N}}.
\end{equation}
\end{thm}
Moreover, the same arguments yield the following result that improves  Theorem 3 in \cite{dyda remder}. 

\begin{thm}\label{Geom-convex-Hardy-bndd+}
Let $\Omega\subset\R^N$ be a bounded convex open set. Then
\begin{equation*}\label{}
\frac{c_{N,s}}{2}\iint_{\Omega\times\Omega}\frac{(u(x)-u(y))^2}{|x-y|^{N+2s}}\ dxdy\geq\mathfrak{h}_{N,s}\int_{\Omega}\frac{u^2}{\delta^{2s}}\ dx+a(N,s)|\Omega|^{-\frac{2s}{N}}\int_{\Omega}u^2\ dx+\mathcal{K}\int_{\Omega}\frac{u^2}{\delta^{2s-1}}\ dx,
\end{equation*}
for all $u\in H^s_0(\Omega)$, where
\begin{equation}
\mathcal{K}=\mathcal{K}(\Omega,N,s)=2^{2s-1}(4-2^{3-2s})\frac{\Gamma(s)\Gamma(\frac{N+2s}{2})}
{\Gamma(1-s)\Gamma(\frac{N+2s-1}{2})}\diam(\Omega)^{-1}.
\end{equation}
\end{thm}
As a consequence of Theorem \ref{Geom-convex-Hardy}, we obtain a lower bound for $\lambda^*(s,\Omega)$. More precisely, the following holds.
\begin{thm}\label{second-main-result}
Let $\Omega$ be an open convex set. Then 
\begin{equation}\label{second-main-result-convex}
\lambda^*(s,\Omega)\geq a(N,s)|\Omega|^{-\frac{2s}{N}}.
\end{equation}
\end{thm}
We stress that \eqref{second-main-result-convex} is not true in general for non-convex open sets. However, for a bounded open set of class $C^{1,1}$, using that $\mathfrak{h}_{N,s}=0$\footnote{Notice that for $s\in\left(0,\frac{1}{2}\right]$ the space $H^s_0(\Omega)$ contains the characteristic function $1_\Omega$, see \cite{fall2023existence}, and then we also have $\mu_{N,s}(\Omega)=0$.} for $s=\frac{1}{2}$, we were able to prove that $\lambda^*(s,\Omega)\geq 0$ when ``$s$ is close to $\frac{1}{2}$". More precisely, we have the following asymptotic result.
\begin{thm}\label{thm_lmada_asym}
Let $\Omega$ be a bounded domain of class $C^{1,1}$ in $\R^N$, $N\geq2$. Then there exists $s_*=s_*(\Omega)>\frac12$ such that
\begin{equation*}
\lambda^*(s,\Omega)\geq 0 \quad\text{ for all  } s\in\left(\frac 12,\, s_*(\Omega)\right).
\end{equation*}
Moreover,
\begin{align*}
\lim_{s\to \frac{1}{2}^+}\lambda^*(s,\Omega)=0.
\end{align*}
\end{thm}
We prove Theorem \ref{thm_lmada_asym} by using a continuity argument and by exploiting Theorem \ref{thm_hardy_asym} below.

Before stating our final result, we make some comments and recall related results available in the literature. It is well known that, in the classical Hardy inequality with singular weight concentrated on the boundary, the optimal Hardy constant $\mu_{N,1}(\Omega)$ associated with an open set $\Omega$ satisfies 
$$
\mu_{N,1}(\Omega)\leq \mathfrak{h}_{N,1}=\frac14.\footnote{Here $\mathfrak{h}_{N,1}:=\mu_{N,1}(\R^N_+)$ denotes the Hardy constant of the half-space.}
$$
Moreover, the topology and geometry of $\Omega$ play a crucial role in determining whether the equality
$$
\mu_{N,1}(\Omega)=\frac14,
$$
holds. For instance, this is the case for convex domains, see \cite{Dav94,marcus1998best, mat}, and, more generally, for weakly mean convex domains; see \cite{yanli}. The same conclusion also holds for annular domains and for a large class of simply connected domains; see \cite{avkh,bali,LaSo}. On the other hand, the strict inequality also holds for a large family of nonconvex domains, see \cite{LaSo,marcus1998best, Tidblom}.\\

In the fractional setting, the only result currently available in this direction, see \cite{bogdan2011best,loss2010hardy}, states that for convex open sets,
$$
\mu_{N,s}(\Omega)=\mathfrak{h}_{N,s},\quad \text{ for all } s\in\left(\frac{1}{2},\,1\right).
$$
\par
Our next result shows, rather surprisingly, that for $s$ close to $\frac12$, neither the topology nor the geometry of the open set $\Omega$ affects the value of the optimal constant. More precisely, the fractional Hardy constant becomes independent of $\Omega$, exhibiting a behavior markedly different from the local case. To the best of our knowledge, such a phenomenon has not been observed before in the nonlocal setting. Our result reads as follows 
\begin{thm}\label{thm_hardy_asym}
Let $\Omega$ be a bounded domain of class $C^{1,1}$ in $\R^N$, $N\geq2$. Then there exists $s_*=s_*(\Omega)>\frac12$ such that
\begin{align*}
\mu_{N,s}(\Omega)=\mathfrak{h}_{N,s},\quad \text{for all } \frac12 < s \leq s_*.
\end{align*}
In particular, $\mu_{N,s}(\Omega)$ is never achieved for $s$ sufficiently close to $\frac12$.
\end{thm}
The proof relies on a contradiction argument where we have used a general fractional Hardy inequality, Green function estimates for the regional fractional Laplacian\footnote{ This is the reason why the assumption $\Omega \in C^{1,1}$ is required.} and an embedding result for fractional Sobolev spaces.\\

The paper is organized as follows. In Section \ref{section:proof of first main result}, we start by establishing the sharp bound $\mathfrak{J}_{\lambda,s}(\Omega)\leq\mathfrak{h}_{N,s}$ for all $\lambda$. This is contained in Lemma \ref{lema1}. In Lemma \ref{J achieved}, we prove $\mathfrak{J}_{\lambda,s}(\Omega)=\mathfrak{h}_{N,s}$ for some $\lambda$. Next, in Lemma \ref{lema2}, we prove that the infimum in \eqref{a1} is achieved for any $\lambda>\lambda^*(s,\Omega)$, and in Lemma \ref{lema3} we prove the non-attainment of \eqref{a1} for $\lambda\leq\lambda^*(s,\Omega)$. Theorem \ref{first-main-result} is a consequence of these results. This section also contains the proof of Theorem \ref{non-existence-theorem} and Corollary \ref{first-main-result-corl}. In Section \ref{section:proof of second main result}, we prove Theorems \ref{Geom-convex-Hardy}~\&~\ref{second-main-result} using a new geometric improvement of the fractional Hardy inequality, Theorem \ref{Geom-Hardy}. This section also contains the proof of Theorems \ref{Geom-convex-Hardy-bndd}~\&~\ref{Geom-convex-Hardy-bndd+}. Section \ref{main-asym}, contains the proof of Theorems~\ref{thm_lmada_asym}~\&~\ref{thm_hardy_asym}. We conclude the paper in the Appendix, where we prove an asymptotic estimate, namely \eqref{asymptotic-k} and an embedding Theorem.

\section{Proof of Theorem \ref{first-main-result}}\label{section:proof of first main result}
The goal of this section is to prove Theorem \ref{first-main-result}. For this, we first collect some useful results. We start with the following sharp inequality.
\begin{lemma} \label{lema1}
Let $\Omega$ be an open set in $\R^N$ of class $C^{1,1}$. Then
\begin{equation}\label{upper-bound}
\mathfrak{J}_{\lambda,s}(\Omega)\leq\mathfrak{h}_{N,s}.
\end{equation}
for all $\lambda\in \R$.
\end{lemma}
\begin{proof}
To prove \eqref{upper-bound}, we adapt the proof of Theorem 5 in Marcus-Mizel-Pinchover \cite{marcus1998best} for the local case. We first note that 
\begin{equation}\label{b0}
\mathfrak{J}_{\lambda,s}(\Omega)\leq\frac{\frac{c_{N,s}}{2}\displaystyle\iint_{\Omega\times\Omega}\frac{(\phi(x)-\phi(y))^2}{|x-y|^{N+2s}}\ dxdy+|\lambda|\displaystyle\int_{\Omega}\phi(x)^2\, dx}{\displaystyle\int_{\Omega}\frac{\phi(x)^2}{\delta(x)^{2s}}\ dx},~~\forall\lambda\in\R.
\end{equation}
Since $\Omega$ is of class $C^1$, then any point at the boundary has a tangent hyperplane.  Fix $P\in\partial\Omega$ and let $\Pi$ be the tangent hyperplane at $\partial\Omega$ in $P$. Then arguing as in the proof of \cite[Theorem 4.1]{lamberti2019lp}, we get that condition $(2.2)$ in \cite[Theorem 5]{marcus1998best} is satisfied, that is, there exists a neighborhood $U$ of $P$ such that
 \begin{equation}\label{condition}
 |\dist(x,\Pi)-\delta(x)|\leq o(1)\dist(x,P),~~\forall x\in U\cap\Omega
 \end{equation}
where $o(1)\to0$ as $x\to P$, and $U\cap\Omega$ contains segment $\overline{PQ}$ perpendicular to $\Pi$.

Without loss of generality, we can assume that $P=O$, $\Pi=\{x\in\R^N: x_N=0\}$ and that $\Omega$ contains a segment $\{(0,x_N): 0<x_N<b\}$. \\
Let $\varepsilon>0$, since $\mu_{N,s}(\R^N_+)=\mathfrak{h}_{N,s}$, there exists $\phi\in C^{\infty}_c(\R^N_+)$ such that
\begin{equation}\label{b1}
\frac{\frac{c_{N,s}}{2}\displaystyle\iint_{\R^N_+\times\R^N_+}\frac{(\phi(x)-\phi(y))^2}{|x-y|^{N+2s}}\ dxdy}{\int_{\R^N_+}\frac{\phi^2}{x_N^{2s}}\ dx}<\mathfrak{h}_{N,s}+\varepsilon.
\end{equation}
Moreover, there exists $R>0$ such that 
\begin{equation*}
\supp~\phi\subset \mathcal{C}_R=\{x=(\bar x,\,x_N):\; x_N>0, |\bar x|<R\,x_N\}. 
\end{equation*}
In view of \eqref{condition} and by the scale invariance of the left-hand side of \eqref{b1} and $\mathcal{C}_R$ we may assume that, for $\varepsilon>0$ sufficiently small,
\begin{align}\label{cond+}
\supp~\phi\subset U\cap\Omega\,~~~ \text{ and }~~~ \delta(\varepsilon x)\leq (1+o(1))~\varepsilon x_N,\quad \text{for all } x\in\supp~\phi.
 \end{align}
 Here, $o(1)$ is such that $o(1)\to0$ as $x\to O$. We now define the rescaled function
\begin{equation*}
\phi_{\varepsilon}(x)=\phi\left(\frac{x}{\varepsilon}\right).
\end{equation*}
Then, it holds that
\begin{align*}
\int_{\Omega}\displaystyle\frac{\phi_{\varepsilon}(x)^2}{\delta(x)^{2s}}\,dx=
\varepsilon^{N}\int_{U\cap\Omega}\frac{\phi(x)^2}{\delta(\varepsilon x)^{2s}}\ dx.
\end{align*}
Thus, by using \eqref{cond+}, it follows that
\begin{equation}\label{b2}
\int_{\Omega}\frac{\phi_{\varepsilon}(x)^2}{\delta(x)^{2s}}\ dx\geq (1+o(1))\varepsilon^{N-2s}\int_{\R^N_+}\frac{\phi(x)^2}{x_N^{2s}}\ dx.
\end{equation}
On the other hand,
\begin{align}\label{b30}
\int_{\Omega}\phi_{\varepsilon}(x)^2\ dx=\int_{\R^N_+}\phi_{\varepsilon}(x)^2\ dx=\varepsilon^{N}\int_{\R^N_+}\phi(x)^2\, dx
\end{align}
and
\begin{align}\label{b3}
\iint_{\R^N_+\times\R^N_+}\frac{(\phi_{\varepsilon}(x)-\phi_{\varepsilon}(y))^2}{|x-y|^{N+2s}}\ dxdy=\varepsilon^{N-2s}\iint_{\R^N_+\times\R^N_+}\frac{(\phi(x)-\phi(y))^2}{|x-y|^{N+2s}}\ dxdy.
\end{align}
So, from \eqref{b2}, \eqref{b30}, and \eqref{b3}, and recalling \eqref{b0}, we get
\begin{align*}
\mathfrak{J}_{\lambda,s}(\Omega)&\leq\frac{\frac{c_{N,s}}{2}\displaystyle\iint_{\R^N_+\times\R^N_+}\frac{(\phi_{\varepsilon}(x)-\phi_{\varepsilon}(y))^2}{|x-y|^{N+2s}}\ dxdy+|\lambda|\int_{\Omega}\phi_{\varepsilon}(x)^2\ dx}{\displaystyle\int_{\Omega}\frac{\phi_{\varepsilon}(x)^2}{\delta(x)^{2s}}\ dx}\\
&\leq(1+o(1))\frac{\frac{c_{N,s}}{2}\displaystyle\iint_{\R^N_+\times\R^N_+}\frac{(\phi(x)-\phi(y))^2}{|x-y|^{N+2s}}\ dxdy+|\lambda|\varepsilon^{2s}\int_{\R_+^N}\phi(x)^2\ dx}{\displaystyle\int_{\R^N_+}\frac{\phi(x)^2}{x_N^{2s}}\ dx}\\
&\leq(1+o(1))~(\mathfrak{h}_{N,s}+\varepsilon)+o(\varepsilon^{2s}),
\end{align*}
where, in the latter, we used \eqref{b1}. Since $\varepsilon$ can be chosen arbitrarily small, \eqref{upper-bound} follows. This concludes the proof of the Lemma.
\end{proof}

\begin{lemma}\label{J achieved}
Let $\Omega$ be an open bounded set in $\R^N$ of class $C^{1,1}$. Then there exists $\lambda\in\R$ such that
\begin{equation}\label{lower-bound-j}
\mathfrak{J}_{\lambda,s}(\Omega)= \mathfrak{h}_{N,s}.
\end{equation}
\end{lemma}
\begin{proof}
In view of Lemma \ref{lema1}, we only need to show the existence of $\lambda\in\R$ such that
\begin{equation}\label{lower-bound-jj}
\mathfrak{J}_{\lambda,s}(\Omega)\geq \mathfrak{h}_{N,s}.
\end{equation}
To this end, we define for $\beta>0$ sufficiently small (say $\beta<\bar\beta$), 
\begin{equation*}
\Omega_{\beta}:=\{x\in\Omega: \delta(x)<\beta\}.
\end{equation*}

Note that for every $x\in\Omega_{\beta}$, there exists a unique point $\sigma(x)\in\partial\Omega$ such that $\delta(x)=|x-\sigma(x)|$. As in \cite{djitte} we also define $\Psi: \partial \Omega \times (0, \bar\beta) \to \Omega_{\bar\beta}$ as the map 
\begin{equation}\label{def:localCoordBoundary}
\Psi(\sigma,r):= \sigma - r \nu(\sigma),
\end{equation}
where $\nu: \partial \Omega \to \R^N$ is the outward unit  normal vector field. Since  $\partial \Omega\in C^{1,1}$ by assumption, the map $\Psi$ is Lipschitz, moreover for  $\bar\beta>0$ sufficiently small it is  bi-Lipschitz.  In particular, $\Psi$ is a.e. differentiable, with 
\begin{align}\label{diffeo-dom}
\Psi(\partial \Omega \times (0, \bar\beta))= \Omega_{\bar\beta},\; \text{  and  }\;\delta(\Psi(\sigma,t)) = t,
\end{align}
 for $\sigma \in \partial \Omega, 0 \le t < \bar\beta$. Furthermore 
\begin{equation}\label{jacobian-bounds}
\|{\rm Jac}_{{\Psi}}\|_{L^\infty(\Omega_{\bar\beta})}< +\infty\quad \text{and}\quad |{\rm Jac}_{{\Psi}}(\sigma,t)-1|\leq c t \quad \text{for a.e. $\sigma \in \partial \Omega$},
\end{equation}
where $c$ is a constant depending only on $\partial\Omega$, $\bar\beta$, and the choice of local coordinates on $\partial\Omega$. For every $u\in L^1(\Omega_{\beta})$,
\begin{equation}\label{change-of-variable}
    \int_{\Omega_{\beta}}u(x)\ dx=\int_{0}^{\beta}\int_{\partial\Omega}u(\Psi(\sigma,t)){\rm Jac}_{{\Psi}}(\sigma,t)\ d\sigma dt,
\end{equation}
where $d\sigma$ denotes the surface element on $\partial\Omega$ (see e.g. \cite{brezis1997hardy}). In the same way as in \eqref{change-of-variable}, we also have
\begin{align*}\label{ar}
[u]_{s,\Omega_\beta}:=&\frac{c_{N,s}}{2}\iint_{\Omega_{\beta}\times\Omega_{\beta}}\frac{(u(x)-u(y))^2}{|x-y|^{N+2s}}\ dxdy\\
&=\frac{c_{N,s}}{2}\iint_{\partial\Omega\times\partial\Omega}\int_0^\beta\int_0^\beta\frac{(u(\Psi(\sigma,t))-u(\Psi(\bar\sigma,r)))^2}{|\Psi(\sigma,t)-\Psi(\bar\sigma,r)|^{N+2s}}\, {\rm Jac}_{{\Psi}}(\sigma,t){\rm Jac}_{{\Psi}}(\bar\sigma,r) dtdr\,d\sigma d\bar\sigma.
\end{align*}
Using that $\Psi$ is Lipschitz, and a scaling, we obtain
\begin{align*} 
&[u]_{s,\Omega_\beta}
\geq \frac{c_{N,s}}{2}\iint_{\partial\Omega\times\partial\Omega}\int_0^\beta\int_0^\beta\frac{(u(\Psi(\sigma,t))-u(\Psi(\bar\sigma,r)))^2}{|t-r|^{N+2s}}\, C(\bar\beta){\rm Jac}_{{\Psi}}(\sigma,t){\rm Jac}_{{\Psi}}(\bar\sigma,r) dtdr\,d\sigma d\bar\sigma\\
&\geq \frac{c_{N,s}}{2}\frac{C(\bar\beta)}{\beta^{N-2+2s}}\iint_{\partial\Omega\times\partial\Omega}\int_0^1\int_0^1\frac{(u(\Psi(\sigma,\beta t))-u(\Psi(\bar\sigma,\beta r)))^2}{|t-r|^{N+2s}}\, {\rm Jac}_{{\Psi}}(\sigma,\beta t){\rm Jac}_{{\Psi}}(\bar\sigma,\beta r) dtdr\,d\sigma d\bar\sigma\\
&\geq \frac{c_{N,s}}{2}\frac{C(\bar\beta)}{\beta^{N-2+2s}}\iint_{\partial\Omega\times\partial\Omega}\int_0^1\int_0^1\frac{(u(\Psi(\sigma,\beta t))-u(\Psi(\bar\sigma,\beta r)))^2}{|t-r|^{1+2s}}\, {\rm Jac}_{{\Psi}}(\sigma,\beta t){\rm Jac}_{{\Psi}}(\bar\sigma,\beta r) dtdr\,d\sigma d\bar\sigma,
\end{align*}
where we have used, in the last line, that $|t-s|<1$ which implies that $|t-s|^{N+2s}\leq |t-s|^{1+2s}$. By the one-dimensional fractional Hardy inequality, see \cite[Theorem 2.5]{loss2010hardy} , \eqref{diffeo-dom} and \eqref{jacobian-bounds}, we have
\begin{align*}
[u]_{s,\Omega_\beta}
&\geq c_{N,s}\kappa_{1,2s}\frac{C(\bar\beta)}{\beta^{N-2+2s}}(1+o(1))\iint_{\partial\Omega\times\partial\Omega}\int_0^1\frac{(u(\Psi(\sigma,\beta t)))^2}{t^{2s}}\,  dt\,d\sigma d\bar\sigma\\
&\geq
c_{N,s}\kappa_{1,2s}\frac{C(\bar\beta)}{\beta^{N-2}}(1+o(1))\iint_{\partial\Omega\times\partial\Omega}\int_0^1\frac{(u(\Psi(\sigma,\beta t)))^2}{\delta(\Psi(\sigma,\beta t))^{2s}}\,  dt\,d\sigma d\bar\sigma\\
&\geq c_{N,s}\kappa_{1,2s}\frac{C(\bar\beta)}{\beta^{N-2}}(1+o(1))\iint_{\partial\Omega\times\partial\Omega}\int_0^1\frac{(u(\Psi(\sigma,\beta t)))^2}{\delta(\Psi(\sigma,\beta t))^{2s}}\frac{{\rm Jac}_{{\Psi}}(\sigma,\beta t)}{{\rm Jac}_{{\Psi}}(\sigma,\beta t)}\,  dt\,d\sigma d\bar\sigma\\
&\geq c_{N,s}\kappa_{1,2s}\frac{C(\bar\beta)}{\beta^{N-2}}\frac{(1+o(1))}{\|{\rm Jac}_{{\Psi}}\|_{L^\infty(\Omega_\beta)}}\iint_{\partial\Omega\times\partial\Omega}\int_0^1\frac{(u(\Psi(\sigma,\beta t)))^2}{\delta(\Psi(\sigma,\beta t))^{2s}}{\rm Jac}_{{\Psi}}(\sigma,\beta t)\,  dt\,d\sigma d\bar\sigma,
\end{align*}
where  $o(1)\to 0$ as $\beta\to 0$. Finally, we have shown that
\begin{align*}
\frac{c_{N,s}}{2}\iint_{\Omega_{\beta}\times\Omega_{\beta}}\frac{(u(x)-u(y))^2}{|x-y|^{N+2s}}\ dxdy\geq
c_{N,s}\kappa_{1,2s}\frac{C(\bar\beta)}{\beta^{N-1}}\frac{(1+o(1))|\partial\Omega|}{\|{\rm Jac}_{{\Psi}}\|_{L^\infty(\Omega_{\bar\beta})}}\int_{\Omega_\beta}\frac{|u(x)|^2}{\delta(x)^{2s}}\, dx. 
\end{align*}
We may now choose $\beta$ sufficiently small such that 
\begin{equation*}
\frac{C(\bar\beta)}{\beta^{N-1}}\frac{(1+o(1))|\partial\Omega|}{\|{\rm Jac}_{{\Psi}}\|_{L^\infty(\Omega_{\bar\beta})}}\geq \frac{\pi^{\frac{N-1}{2}}\Gamma(\frac{1+2s}{2})}{\Gamma(\frac{N+2s}{2})}.
\end{equation*}
Hence,
\begin{align}\label{a100}
  \nonumber  \frac{c_{N,s}}{2}\iint_{\Omega_{\beta}\times\Omega_{\beta}}\frac{(u(x)-u(y))^2}{|x-y|^{N+2s}}\ dxdy&\geq c_{N,s}\frac{\pi^{\frac{N-1}{2}}\Gamma(\frac{1+2s}{2})}{\Gamma(\frac{N+2s}{2})}\kappa_{1,2s}\int_{\Omega_\beta}\frac{|u(x)|^2}{\delta(x)^{2s}}\, dx\\
    &=c_{N,s}\kappa_{N,2s}\int_{\Omega_\beta}\frac{|u(x)|^2}{\delta(x)^{2s}}\, dx=\mathfrak{h}_{N,s}\int_{\Omega_\beta}\frac{|u(x)|^2}{\delta(x)^{2s}}\, dx. 
\end{align}

Next, take $\varrho\in C^{\infty}(\Omega)$ such that $0\leq\varrho\leq1$, $\varrho\equiv1$ in $\Omega_{\frac{\beta}{4}}$ and $\varrho\equiv0$ in $\Omega\setminus\Omega_{\frac{\beta}{2}}$. Then, for $u\in H^s_0(\Omega)$,
    \begin{align}\label{a131}
   \nonumber \int_{\Omega}\frac{u^2}{\delta^{2s}}\ dx&=\int_{\Omega_{\frac{\beta}{4}}}\frac{u^2}{\delta^{2s}}\ dx+\int_{\Omega\setminus\Omega_{\frac{\beta}{4}}}\frac{u^2}{\delta^{2s}}\ dx\\
    \nonumber&=\int_{\Omega_{\frac{\beta}{4}}}\frac{(\varrho u)^2}{\delta^{2s}}\ dx+\int_{\Omega\setminus\Omega_{\frac{\beta}{4}}}\frac{u^2}{\delta^{2s}}\ dx\\
    &\leq \int_{\Omega}\frac{(\varrho u)^2}{\delta^{2s}}\ dx+\left(\frac{4}{\beta}\right)^{2s}\int_{\Omega}u^2\ dx.
    \end{align}
Now, by \eqref{a100}, we have 
    \begin{equation}\label{a15}
    \int_{\Omega}\frac{(\varrho u)^2}{\delta^{2s}}\ dx\leq\frac{1}{\mathfrak{h}_{N,s}}\frac{c_{N,s}}{2}\iint_{\Omega\times\Omega}\frac{((\varrho u)(x)-(\varrho u)(y))^2}{|x-y|^{N+2s}}\ dxdy.
    \end{equation}
Using the Elementary Identity
    \begin{equation*}
    ((\varrho u)(x)-(\varrho u)(y))^2=\varrho(x)\varrho(y)(u(x)-u(y))^2+(\varrho(x)-\varrho(y))(\varrho(x)u(x)^2-\varrho(y)u(y)^2),
    \end{equation*}
we have 
\begin{align}\label{a16}
\nonumber&\iint_{\Omega\times\Omega}\frac{((\varrho u)(x)-(\varrho u)(y))^2}{|x-y|^{N+2s}}\ dxdy\\
\nonumber&=\iint_{\Omega\times\Omega}\varrho(x)\varrho(y)\frac{(u(x)-u(y))^2}{|x-y|^{N+2s}}\ dxdy +\iint_{\Omega\times\Omega}\frac{(\varrho(x)-\varrho(y))}{|x-y|^{N+2s}}(\varrho(x)u(x)^2-\varrho(y)u(y)^2)\ dxdy\\
&\leq\iint_{\Omega\times\Omega}\frac{(u(x)-u(y))^2}{|x-y|^{N+2s}}\ dxdy +2\iint_{\Omega\times\Omega}\varrho(x)u(x)^2\frac{(\varrho(x)-\varrho(y))}{|x-y|^{N+2s}}\ dxdy.
\end{align}
We estimate the second term on the right-hand side of \eqref{a16} as 
\begin{align}\label{a17}
    \nonumber&\iint_{\Omega\times\Omega}\varrho(x)u(x)^2\frac{(\varrho(x)-\varrho(y))}{|x-y|^{N+2s}}\ dxdy\\
    \nonumber&=\iint_{|x-y|\geq1}\varrho(x)u(x)^2\frac{(\varrho(x)-\varrho(y))}{|x-y|^{N+2s}}\ dxdy+\iint_{|x-y|<1}\varrho(x)u(x)^2\frac{(\varrho(x)-\varrho(y))}{|x-y|^{N+2s}}\ dxdy\\
    \nonumber&\leq2|\Omega|\int_{\Omega}u^2\ dx+\int_{\Omega}u(x)^2\ dx\int_{B_1(x)}\frac{(\varrho(x)-\varrho(y))}{|x-y|^{N+2s}}\ dy\\
    \nonumber&\leq2|\Omega|\int_{\Omega}u^2\ dx+\frac{1}{2}\int_{\Omega}u(x)^2\ dx\int_{B_1}\frac{2\varrho(x)-\varrho(x+z)-\varrho(x-z)}{|z|^{N+2s}}\ dz\\
    \nonumber&\leq2|\Omega|\int_{\Omega}u^2\ dx+\frac{\|D^2\varrho\|_{L^{\infty}(B_1)}}{2}\int_{\Omega}u(x)^2\ dx\int_{B_1}|z|^{2-N-2s}\ dz\\
    &=\left(2|\Omega|+\frac{\|D^2\varrho\|_{L^{\infty}(B_1)}|\mathbb{S}^{N-1}|}{4(1-s)}\right)\int_{\Omega}u^2\ dx.
\end{align}
Thus, plugging \eqref{a17} into \eqref{a16}, we get 
    \begin{align}\label{a18}
    \nonumber&\iint_{\Omega\times\Omega}\frac{((\varrho u)(x)-(\varrho u)(y))^2}{|x-y|^{N+2s}}\ dxdy\\
    &\leq\iint_{\Omega\times\Omega}\frac{(u(x)-u(y))^2}{|x-y|^{N+2s}}\ dxdy+\left(4|\Omega|+\frac{\|D^2\varrho\|_{L^{\infty}(B_1)}|\mathbb{S}^{N-1}|}{2(1-s)}\right)\int_{\Omega}u^2\ dx.
    \end{align}
    Hence, it follows from \eqref{a15} that 
     \begin{equation}\label{a19}
    \int_{\Omega}\frac{(\varrho u)^2}{\delta^{2s}}\ dx\leq\frac{1}{\mathfrak{h}_{N,s}}\frac{c_{N,s}}{2}\iint_{\Omega\times\Omega}\frac{(u(x)-u(y))^2}{|x-y|^{N+2s}}\ dxdy+c_1\int_{\Omega}u^2\ dx,
    \end{equation}
    where $c_1=\frac{1}{\mathfrak{h}_{N,s}}\frac{c_{N,s}}{2}\left(2|\Omega|+\frac{\|D^2\varrho\|_{L^{\infty}(B_1)}|\mathbb{S}^{N-1}|}{2(1-s)}\right)$. Therefore, \eqref{a19} and \eqref{a131} yield
    \begin{equation*}
    \int_{\Omega}\frac{u^2}{\delta^{2s}}\ dx\leq\frac{1}{\mathfrak{h}_{N,s}}\frac{c_{N,s}}{2}\iint_{\Omega\times\Omega}\frac{(u(x)-u(y))^2}{|x-y|^{N+2s}}\ dxdy+(c_1+c_2)\int_{\Omega}u^2\ dx,
    \end{equation*}
    that is,
    \begin{equation*}
    \frac{c_{N,s}}{2}\iint_{\Omega\times\Omega}\frac{(u(x)-u(y))^2}{|x-y|^{N+2s}}\ dxdy\geq \mathfrak{h}_{N,s}\int_{\Omega}\frac{u^2}{\delta^{2s}}\ dx+\lambda\int_{\Omega}u^2\ dx
    \end{equation*}
    where $\lambda=-\mathfrak{h}_{N,s}(c_1+c_2)$ with $c_2=\left(\frac{4}{\beta}\right)^{2s}$. Thus, \eqref{lower-bound-jj} is established. This completes the proof. 
\end{proof}

\begin{remark}
	Define 
	\begin{equation}
	\lambda^*(s,\Omega)=\sup\{\lambda\in\R: \mathfrak{J}_{\lambda,s}(\Omega)=\mathfrak{h}_{N,s}\}.
	\end{equation}
	Then \eqref{a6} and \eqref{a7} are valid. As mentioned in the introduction, the function $\lambda\mapsto\mathfrak{J}_{\lambda,s}(\Omega)$ is concave and non-increasing on $\R$ and $\lim_{\lambda\to+\infty}\mathfrak{J}_{\lambda,s}(\Omega)=-\infty$. This fact and Lemma \ref{lema1} imply that $-\infty<\lambda^*(s,\Omega)<\infty$.
\end{remark}
We have the following existence result. 
\begin{lemma} \label{lema2}
Let $\Omega$ be a bounded open set in $\R^N$ of class $C^{1,1}$. Then any minimizing sequence for $\mathfrak{J}_{\lambda,s}(\Omega)$, normalized in $L^2(\Omega; \delta^{-2s})$, is relatively compact in $H^s_0(\Omega)$ for every $\lambda>\lambda^*(s,\Omega)$. In particular, the infimum is achieved. 
\end{lemma}

\begin{proof}
	Let $\{u_k\}$ be a minimizing sequence for $\mathfrak{J}_{\lambda,s}(\Omega)$, normalized so that
	\begin{equation}\label{normalized-condition}
	\int_{\Omega}\frac{u_k^2}{\delta^{2s}}\ dx=1~~~\forall k.
	\end{equation}
Then the sequence is bounded in $H^s_0(\Omega)$ and consequently, up to a subsequence, there exists $u\in H^s_0(\Omega)$ such that
\begin{align}\label{d1}
u_k\rightharpoonup u~~\text{weakly in}~H^s_0(\Omega),~
u_k\to u~~\text{strongly in}~L^2(\Omega),~\text{and}~u_k\to u~~\text{a.e. in}~\Omega.
\end{align}
Set $v_k=u_k-u$. From \eqref{d1}, we have that
\begin{equation}\label{d2}
v_k\to 0~~\text{strongly in}~L^2(\Omega),~v_k\rightharpoonup 0~~ \text{weakly in}~H^s_0(\Omega),~ \text{and}~\frac{v_k}{\delta}\rightharpoonup 0~~ \text{weakly in}~ L^2(\Omega).
\end{equation}
Notice that the latter is a direct consequence of the fact that $\{\frac{v_k}{\delta}\}$ is bounded in $L^2(\Omega)$ and strongly converges to zero in $L^2_{loc}(\Omega)$. Using \eqref{d1} and \eqref{d2}, we get
\begin{align}\label{d3}
\nonumber \mathfrak{J}_{\lambda,s}(\Omega)+o(1)&=\frac{c_{N,s}}{2}\iint_{\Omega\times\Omega}\frac{(u_k(x)-u_k(y))^2}{|x-y|^{N+2s}}\ dxdy-\lambda\int_{\Omega}u_k^2\ dx\\
&=\frac{c_{N,s}}{2}\iint_{\Omega\times\Omega}\frac{(u(x)-u(y))^2}{|x-y|^{N+2s}}\ dxdy-\lambda\int_{\Omega}u^2\ dx\\
\nonumber&~~~~~~~+\frac{c_{N,s}}{2}\iint_{\Omega\times\Omega}\frac{(v_k(x)-v_k(y))^2}{|x-y|^{N+2s}}\ dxdy+o(1),
\end{align}
and
\begin{equation}\label{d4}
1=\int_{\Omega}\frac{u_k^2}{\delta^{2s}}\ dx=\int_{\Omega}\frac{u^2}{\delta^{2s}}\ dx+\int_{\Omega}\frac{v_k^2}{\delta^{2s}}\ dx+o(1).
\end{equation}
Let $\overline{\lambda}<\lambda^*(s,\Omega)$ so that $\mathfrak{J}_{\ov\lambda,s}(\Omega)=\mathfrak{h}_{N,s}$. Then 
\begin{equation}\label{d5}
\frac{c_{N,s}}{2}\iint_{\Omega\times\Omega}\frac{(v_k(x)-v_k(y))^2}{|x-y|^{N+2s}}\ dxdy-\overline{\lambda}\int_{\Omega}v_k^2\ dx\geq\mathfrak{h}_{N,s}\int_{\Omega}\frac{v_k^2}{\delta^{2s}}\ dx.
\end{equation}
By \eqref{d2} and \eqref{d4}, it follows from \eqref{d5} that
\begin{equation*}
\frac{c_{N,s}}{2}\iint_{\Omega\times\Omega}\frac{(v_k(x)-v_k(y))^2}{|x-y|^{N+2s}}\ dxdy\geq\mathfrak{h}_{N,s}\left(1-\int_{\Omega}\frac{u^2}{\delta^{2s}}\ dx\right)+o(1).
\end{equation*}
From this inequality and \eqref{d3}, we get
\begin{equation*}
\frac{c_{N,s}}{2}\iint_{\Omega\times\Omega}\frac{(u(x)-u(y))^2}{|x-y|^{N+2s}}\ dxdy+\mathfrak{h}_{N,s}\left(1-\int_{\Omega}\frac{u^2}{\delta^{2s}}\ dx\right)-\lambda\int_{\Omega}u^2\ dx\leq \mathfrak{J}_{\lambda,s}(\Omega),
\end{equation*}
which, together with (thanks to the definition of $\mathfrak{J}_{\lambda,s}(\Omega)$) 
\begin{equation*}
\frac{c_{N,s}}{2}\iint_{\Omega\times\Omega}\frac{(u(x)-u(y))^2}{|x-y|^{N+2s}}\ dxdy-\lambda\int_{\Omega}u^2\ dx\geq \mathfrak{J}_{\lambda,s}(\Omega)\int_{\Omega}\frac{u^2}{\delta^{2s}}\ dx,
\end{equation*}
yield
\begin{equation*}
\Big(\mathfrak{J}_{\lambda,s}(\Omega)-\mathfrak{h}_{N,s}\Big)\left(\int_{\Omega}\frac{u^2}{\delta^{2s}}\ dx-1\right)\leq0. 
\end{equation*}
Since $\mathfrak{J}_{\lambda,s}(\Omega)<\mathfrak{h}_{N,s}$ and $\displaystyle\int_{\Omega}\frac{u^2}{\delta^{2s}}\ dx\leq1$ we deduce from the above inequality that $\displaystyle\int_{\Omega}\frac{u^2}{\delta^{2s}}\ dx=1$. Therefore, by \eqref{d3}, $u$ is a minimizer for $\mathfrak{J}_{\lambda,s}(\Omega)$ and $v_k\to0$ strongly in $H^s_0(\Omega)$. Thus, $u_k\to u$ strongly in $H^s_0(\Omega)$.
\end{proof}

We have the following nonexistence result.
\begin{lemma} \label{lema3}
	Let $\Omega$ be a bounded open set in $\R^N$ of class $C^{1,1}$. If $\lambda\leq\lambda^*(s,\Omega)$, the infimum \eqref{a1} is not attained.
\end{lemma}

\begin{proof}
	We consider the following two cases: $\lambda<\lambda^*(s,\Omega)$ and $\lambda=\lambda^*(s,\Omega)$.
	
	Case 1: $\lambda<\lambda^*(s,\Omega)$. We argue by contradiction. Suppose that for some $\overline{\lambda}<\lambda^*(s,\Omega)$, the infimum \eqref{a1} is attained at some $\overline{u}\in H^s_0(\Omega)$. Assume that $\overline{u}$ is normalized so that
	\begin{equation*}
	\int_{\Omega}\frac{\overline{u}^2}{\delta^{2s}}\ dx=1 ~~\text{and}~~\frac{c_{N,s}}{2}\iint_{\Omega\times\Omega}\frac{(\overline{u}(x)-\overline{u}(y))^2}{|x-y|^{N+2s}}\ dxdy-\overline{\lambda}\int_{\Omega}\overline{u}^2\ dx=\mathfrak{h}_{N,s}.
	\end{equation*}
	Then for $\overline{\lambda}<\lambda<\lambda^*(s,\Omega)$, we have
	\begin{align*}
	\mathfrak{h}_{N,s}=\mathfrak{J}_{\lambda,s}(\Omega)&\leq \frac{c_{N,s}}{2}\iint_{\Omega\times\Omega}\frac{(\overline{u}(x)-\overline{u}(y))^2}{|x-y|^{N+2s}}\ dxdy-\lambda\int_{\Omega}\overline{u}^2\ dx\\
	&<\frac{c_{N,s}}{2}\iint_{\Omega\times\Omega}\frac{(\overline{u}(x)-\overline{u}(y))^2}{|x-y|^{N+2s}}\ dxdy-\overline{\lambda}\int_{\Omega}\overline{u}^2\ dx=\mathfrak{h}_{N,s},
	\end{align*}
	which is a contradiction. 
	
	Case 2: $\lambda=\lambda^*(s,\Omega)$. Assume that the infimum $\mathfrak{J}_{\lambda^*_s,s}(\Omega)$ is attained at some $u\in H^s_0(\Omega)$. Then\footnote{Observe that if (for some $\lambda$) the infimum in \eqref{a1} is achieved by a function $w$ then it is also achieved by $|w|$. Thus, $w$ does not change its sign in $\Omega$. In particular, we may assume that $w\geq0$.} $u\geq0$ in $\Omega$ and satisfies
	\begin{equation*}
	(-\Delta)^s_{\Omega}u-\lambda^*_s u=\mathfrak{J}_{\lambda^*_s,s}(\Omega)\frac{u}{\delta^{2s}}~~~\text{in}~~\Omega,
	\end{equation*}
	where $\lambda^*_s=\lambda^*(s,\Omega)$. By the strong maximum principle, $u>0$ in $\Omega$. However, since $\mathfrak{J}_{\lambda^*_s,s}(\Omega)=\mathfrak{h}_{N,s}$, then by Theorem \ref{non-existence-theorem}, we deduce that $u\equiv0$, which is a contradiction.  
\end{proof}

\begin{proof}[Proof of Theorem \ref{non-existence-theorem}]
We adapt the proof of Theorem III in Brezis-Marcus \cite{brezis1997hardy} for the local case. Assume by contradiction that there exists a non-negative function $u$ as stated in the lemma and that $u\neq0$. By the strong maximum principle, $u>0$ in $\Omega$.\\	
Define
\begin{equation*}
L:=(-\Delta)^s_{\Omega}-\frac{\mathfrak{h}_{N,s}}{\delta^{2s}}+\frac{\eta}{\delta^{2s}}
\end{equation*}
so that
\begin{equation}\label{r0}
Lu\geq0~~\text{in}~\Omega.
\end{equation}
Now, for $\alpha>0$, we let $\chi_{\alpha}$ to be the function defined by 
\begin{equation*}
\chi_{\alpha}(t)=\left\{
\begin{aligned}
& t^{\frac{2s-1}{2}}(1-\log t)^{-\alpha},~~~\text{if}~~0<t\leq 1\\
& t^{\frac{2s-1}{2}},~~~~~~~~~~~~~~~~~~~\text{if}~~t>1.
\end{aligned}
\right.
\end{equation*}
Then
	\begin{equation}\label{chi-prime}
	\chi_{\alpha}'(t)=\left\{
	\begin{aligned}
	& \frac{(2s-1)}{2}t^{-\frac{3-2s}{2}}(1-\log t)^{-\alpha}+\alpha t^{-\frac{3-2s}{2}}(1-\log t)^{-\alpha-1},~~~\text{if}~~0<t\leq 1\\
	& \frac{(2s-1)}{2}t^{-\frac{3-2s}{2}},\qquad\qquad\qquad\qquad\qquad~~~~~~~~~~~~~~~~~~~~~~\text{if}~~t>1
	\end{aligned}
	\right.
	\end{equation}
	and 
		\begin{equation}\label{chi-doube-prime}
	\chi_{\alpha}''(t)=\left\{
	\begin{aligned}
	&-\frac{(2s-1)(3-2s)}{4}t^{-\frac{5-2s}{2}}(1-\log t)^{-\alpha}+\alpha(\alpha+1) t^{-\frac{5-2s}{2}}(1-\log t)^{-\alpha-2},~0<t\leq 1\\
	&-\frac{(2s-1)(3-2s)}{4}t^{-\frac{5-2s}{2}},\qquad\qquad\qquad\qquad\qquad\qquad\qquad~~~~~~~~~~~~~~~~~~~~t>1. 
	\end{aligned}
	\right.
	\end{equation}
     Let also $v_{\alpha}: \Omega\to\R$ be the function defined by 
\begin{equation*}
v_{\alpha}=\chi_{\alpha}\circ\delta.   
\end{equation*}
We assume $\alpha>\frac{1}{2}$ such that $v_{\alpha}\in H^s_0(\Omega)$. 
Let $\eta\in C^{\infty}_c(\R^N)$ be such that $\eta\geq0$ in $\R^N$ and $\supp~ \eta\subseteq B_1$. We also assume that $\displaystyle\int_{\R^N}\eta(x)\ dx=1$.  For all $\varepsilon\in(0,1]$, we let $\eta_{\varepsilon}$ be the mollifier defined as 
\begin{equation*}
\eta_{\varepsilon}(x)=\varepsilon^{-N}\eta\left(\frac{x}{\varepsilon}\right)~~~~x\in\R^N.
\end{equation*}
Put $\delta_{\varepsilon}=\eta_{\varepsilon}*\delta$. Then, by construction $\delta_{\varepsilon}$ is a smooth function, i.e., $\delta_{\varepsilon}\in C^{\infty}(\R^N)$ and $\delta_{\varepsilon}\to\delta$ as $\varepsilon\to\infty.$ Next, we introduce the following function
\begin{equation*}
v_{\alpha,\varepsilon}=\chi_{\alpha}\circ\delta_{\varepsilon}.
\end{equation*}
Then $v_{\alpha,\varepsilon}\to v_{\alpha}$ as $\varepsilon\to0$. Notice that $\chi_{\alpha}(0):=\lim_{t\to0}\chi_{\alpha}(t)=0$ and that $(-\Delta)^s_{\Omega}\delta_{\varepsilon}(x)$ exists for all $x\in\Omega$ and $|(-\Delta)^s_{\Omega}\delta_{\varepsilon}(x)|\leq C$ for some $C>0$. Therefore all the assumptions of Lemma 2.3 in \cite{chen2014semilinear} are fulfilled. We can therefore apply this lemma to see that\footnote{A sketch of the proof of Lemma 2.3 in \cite{chen2014semilinear} shows that the same proof still hold if we assume $\Omega$ to be an open set, not necessarily connected, since in our case $\delta(\ov\Omega)$ is an interval for any open set.}
    \begin{equation}\label{chen-veron-identity-n}
(-\Delta)^s_{\Omega}v_{\alpha,\varepsilon}(x)=(\chi_{\alpha}'\circ\delta_\varepsilon)(x)(-\Delta)^s_{\Omega}\delta_\varepsilon(x)-\frac{\chi_{\alpha}''\circ\delta_\varepsilon(z_x)}{2}\int_{\Omega}\frac{(\delta_\varepsilon(y)-\delta_\varepsilon(x))^2}{|y-x|^{N+2s}}\ dy
\end{equation}
for all $x\in\ov\Omega$ and for some $z_x\in\ov\Omega$. Note that 
\begin{equation}\label{q1}
    D^{a}\chi_{\alpha}\circ\delta_{\varepsilon}\to D^{a}\chi_{\alpha}\circ\delta~~\textnormal{strongly in}~\R^N~\textnormal{for all}~ a\in\N_0^{N}.
\end{equation}
Moreover, by dominated convergence theorem,
\begin{equation}\label{q2}
    \int_{\Omega}\frac{(\delta_\varepsilon(y)-\delta_\varepsilon(x))^2}{|y-x|^{N+2s}}\ dy\to \int_{\Omega}\frac{(\delta(y)-\delta(x))^2}{|y-x|^{N+2s}}\ dy~~\textnormal{as}~~\varepsilon\to 0.
\end{equation}
Now, for all $\phi\in C^{\infty}_c(\Omega_{\beta})$,
\begin{align}\label{q3}
   \nonumber &\int_{\Omega_{\beta}}\phi(x) (-\Delta)^s_{\Omega}\delta_{\varepsilon}(x)\ dx=\int_{\Omega}\phi(x) (-\Delta)^s_{\Omega}\delta_{\varepsilon}(x)\ dx=\int_{\Omega}\delta_{\varepsilon}(x) (-\Delta)^s_{\Omega}\phi(x)\ dx\\
    &\to \int_{\Omega}\delta(x) (-\Delta)^s_{\Omega}\phi(x)\ dx=\int_{\Omega}\phi(x) (-\Delta)^s_{\Omega}\delta(x)\ dx=\int_{\Omega_{\beta}}\phi(x) (-\Delta)^s_{\Omega}\delta(x)\ dx.
\end{align}
By setting
\begin{align*}
U_{\alpha,\varepsilon}(x):=(\chi_{\alpha}'\circ\delta_\varepsilon)(x)(-\Delta)^s_{\Omega}\delta_\varepsilon(x)-\frac{\chi_{\alpha}''\circ\delta_\varepsilon(z_x)}{2}\int_{\Omega}\frac{(\delta_\varepsilon(y)-\delta_\varepsilon(x))^2}{|y-x|^{N+2s}}\ dy,
\end{align*}
and
\begin{equation*}
U_{\alpha}(x):=(\chi_{\alpha}'\circ\delta)(x)(-\Delta)^s_{\Omega}\delta(x)-\frac{\chi_{\alpha}''\circ\delta(z_x)}{2}\int_{\Omega}\frac{(\delta(y)-\delta(x))^2}{|y-x|^{N+2s}}\ dy,
\end{equation*}
we infer from \eqref{q1}, \eqref{q2}, \eqref{q3}, and \eqref{chen-veron-identity-n} that 
\begin{equation*}
\int_{\Omega_{\beta}}[U_{\alpha}(x)-(-\Delta)^s_{\Omega}v_{\alpha}(x)]\phi(x)\ dx=0,~~~\forall\phi\in C^{\infty}_c(\Omega_{\beta}).
\end{equation*}   
In particular,
\begin{equation*}
(-\Delta)^s_{\Omega}v_{\alpha}(x)=U_{\alpha}(x)~~~~~\text{ in }\Omega_{\beta}.
\end{equation*}
Recalling \eqref{chi-prime} and \eqref{chi-doube-prime}, the above equation leads to
\begin{align*}
\nonumber&(-\Delta)^s_{\Omega}v_{\alpha}(x)=\frac{(2s-1)}{2}\delta(x)^{-\frac{(3-2s)}{2}}(1-\log \delta(x))^{-\alpha-1}\left(1-\log\delta(x)+\frac{2}{2s-1}\alpha\right) (-\Delta)^s_{\Omega}\delta(x)\\
&+ \frac{(2s-1)(3-2s)}{8}\delta(z_x)^{-\frac{(5-2s)}{2}}(1-\log\delta(z_x))^{-\alpha-2}\Big((1-\log\delta(z_x))^2+\gamma_{\alpha,s}(1-\log\delta(z_x))\Big) K(x)\\
&-\frac{\alpha(\alpha+1)}{2}\delta(z_x)^{-\frac{(5-2s)}{2}}(1-\log\delta(z_x))^{-\alpha-2}K(x),
\end{align*}
where 
\begin{equation*}
\gamma_{\alpha,s}=\frac{8\alpha(1-s)}{(2s-1)(3-2s)}~~~~\textnormal{and}~~~~K(x):=\int_{\Omega}\frac{(\delta(y)-\delta(x))^2}{|y-x|^{N+2s}}\ dy.
\end{equation*}
Consequently, it holds that  
\begin{align*}
&Lv_{\alpha}(x)=(-\Delta)^s_{\Omega}v_{\alpha}(x)-\frac{\mathfrak{h}_{N,s}}{\delta(x)^{2s}}v_{\alpha}(x)+\frac{\eta(x)}{\delta(x)^{2s}}v_{\alpha}(x)\\
&=\frac{(2s-1)}{2}\delta(x)^{-\frac{(3-2s)}{2}}(1-\log \delta(x))^{-\alpha-1}\Big(1-\log\delta(x)+\frac{2}{2s-1}\alpha\Big) (-\Delta)^s_{\Omega}\delta(x)\\
&+ \frac{(2s-1)(3-2s)}{8}\delta(z_x)^{-\frac{(5-2s)}{2}}(1-\log\delta(z_x))^{-\alpha-2}\Big((1-\log\delta(z_x))^2+\gamma_{\alpha,s}(1-\log\delta(z_x))\Big) K(x)\\
&-\frac{\alpha(\alpha+1)}{2}\delta(z_x)^{-\frac{(5-2s)}{2}}(1-\log\delta(z_x))^{-\alpha-2}K(x)+(\eta(x)-\mathfrak{h}_{N,s})\delta(x)^{-\frac{(2s+1)}{2}}(1-\log\delta(x))^{-\alpha}.
\end{align*}
 Since $(-\Delta)^s_{\Omega}\delta$ behaves like $-\delta^{1-2s}$ in $\Omega_\beta$, see \cite[Proposition 2.1]{chen2025liouville} and (see Appendix \ref{appendix2}) 
 \begin{equation}\label{asymptotic-k}
K(x)\sim\delta(x)^{2-2s}~~~\textnormal{as}~~~\delta(x)\to0,
 \end{equation}
assumption \eqref{condition-nonexistence} implies that, for small $\delta$, the dominant term on the right-hand side of the above equation is $\frac{(2s-1)}{2}\delta(x)^{-\frac{(3-2s)}{2}}(1-\log \delta(x))^{-\alpha-1}(1-\log\delta(x)+\frac{2}{2s-1}\alpha) (-\Delta)^s_{\Omega}\delta(x)$ which is non-positive in $\Omega_{\beta}$. Therefore, for sufficiently small $\beta$, independent of $\alpha$, 
	\begin{equation}\label{r2}
	Lv_{\alpha}\leq0~~~\text{in}~~\Omega_{\beta}.
	\end{equation}
	Since $u$ is positive, there exists $\varepsilon>0$ such that for all $\alpha\in(\frac{1}{2},1)$,
	\begin{equation*}
	u\geq\varepsilon v_{\alpha}~~~\text{in}~~\Sigma_{\beta}=\{x\in\Omega: \delta(x)=\beta\}.
	\end{equation*}
	Define $w_{\alpha}=\varepsilon v_{\alpha}-u$. Then, $w_{\alpha}^+\in H^s_0(\Omega_{\beta})$ and $Lw_{\alpha}\leq0$ in $\Omega_{\beta}$, thanks to \eqref{r0} and \eqref{r2}.  Thus,    
	\begin{align}\label{r3}
	\int_{\Omega_{\beta}}Lw_{\alpha}w_{\alpha}^+\ dx=\int_{\Omega_{\beta}}(-\Delta)^s_{\Omega}w_{\alpha}w_{\alpha}^+\ dx-\mathfrak{h}_{N,s}\int_{\Omega_{\beta}}\frac{w_{\alpha}w_{\alpha}^+}{\delta^{2s}}\ dx+\int_{\Omega_{\beta}}\eta\frac{w_{\alpha}w_{\alpha}^+}{\delta^{2s}}\ dx \leq0.
	\end{align}
	Now, utilizing the elementary identity $(a(x)-a(y))(a^+(x)-a^+(y))\geq(a^+(x)-a^+(y))^2$, we have
	\begin{align*}
	\int_{\Omega_{\beta}}(-\Delta)^s_{\Omega}w_{\alpha}w_{\alpha}^+\ dx&=\int_{\Omega}(-\Delta)^s_{\Omega}w_{\alpha}w_{\alpha}^+\ dx\\
    &=\frac{c_{N,s}}{2}\iint_{\Omega\times\Omega}\frac{(w_{\alpha}(x)-w_{\alpha}(y))(w_{\alpha}^+(x)-w_{\alpha}^+(y))}{|x-y|^{N+2s}}\ dxdy\\
	&\geq\frac{c_{N,s}}{2}\iint_{\Omega\times\Omega}\frac{(w_{\alpha}^+(x)-w_{\alpha}^+(y))^2}{|x-y|^{N+2s}}\ dxdy\\
	&\geq \frac{c_{N,s}}{2}\iint_{\Omega_{\beta}\times\Omega_{\beta}}\frac{(w_{\alpha}^+(x)-w_{\alpha}^+(y))^2}{|x-y|^{N+2s}}\ dxdy.
	\end{align*}
Thus, for $\beta$ sufficiently small, 
	\begin{align*}
		\int_{\Omega_{\beta}}(-\Delta)^s_{\Omega}w_{\alpha}w_{\alpha}^+\ dx&-\mathfrak{h}_{N,s} \int_{\Omega_{\beta}}\frac{(w_{\alpha}^+)^2}{\delta^{2s}}\ dx\\
		&\geq \frac{c_{N,s}}{2}\iint_{\Omega_{\beta}\times\Omega_{\beta}}\frac{(w_{\alpha}^+(x)-w_{\alpha}^+(y))^2}{|x-y|^{N+2s}}\ dxdy-\mathfrak{h}_{N,s}\int_{\Omega_{\beta}}\frac{(w_{\alpha}^+)^2}{\delta^{2s}}\ dx. 
	\end{align*}
The right-hand side of the above inequality is nonnegative, thanks to \eqref{a100}. Taking this into account, we deduce from \eqref{r3} that $w_{\alpha}^+=0$, that is, $u\geq\varepsilon v_{\alpha}$ in $\Omega_{\beta}$ for every $\alpha\in (\frac{1}{2},1)$. Thus, $u\geq\varepsilon(\delta^{2s-1}(1-\log\delta)^{-1})^{\frac{1}{2}}$ in $\Omega_{\beta}$ and consequently $\frac{u}{\delta^s}\notin L^2(\Omega_{\beta})$\footnote{Note that, sine $u\in H^s_0(\Omega)$, by fractional Hardy inequality we have $\frac{u}{\delta^s}\in L^2(\Omega)$.}, which contradicts the fact that $u\in H^s_0(\Omega)$. The proof is therefore finished.
\end{proof}
We are now ready to prove Theorem \ref{first-main-result}.

\begin{proof}[Proof of Theorem \ref{first-main-result} (completed)]
The proof follows from Lemmas \ref{lema1}, \ref{J achieved}, \ref{lema2} \& \ref{lema3}.
\end{proof}

\begin{proof}[Proof of Corollary \ref{first-main-result-corl}]
The first part of Corollary \ref{first-main-result-corl} is a straightforward application of Theorem \ref{first-main-result}. Let us now turn our attention to the second part. Let us assume that \eqref{a3} admits a minimizer $u\in H^s_0(\Omega)$, in particular, this implies that $\mu_{N,s}(\Omega)>0$. Since $|u|\in H^s_0(\Omega)$ and 
\[
\mu_{N,s}(\Omega)\le \big[|u|\big]_s\le [u]_s=\mu_{N,s}(\Omega),
\]
we can suppose, without loss of generality,  that $u$ is non-negative. Here we used the notation
$$[u]_s:=\frac{c_{N,s}}{2}\iint_{\Omega\times\Omega}\frac{(u(x)-u(y))^2}{|x-y|^{N+2s}}\ dxdy.$$
\par
Moreover, by standard arguments, $u$ is a minimizer of the functional
\[
\mathcal{I}(\varphi) := \frac{1}{2} \, [ \varphi ]_s- \frac{\mu_{N,s}(\Omega)}{2} \int_{\Omega} \frac{\varphi^2}{\delta^{2s}} \, dx, \qquad \text{ for all } \varphi \in H^{s}_0(\Omega),
\]
and, due to the normalization
\begin{align}\label{normalization*}
        \displaystyle\int_\Omega\frac{\varphi^2}{\delta^{2s}} \, dx=1,
\end{align}
$u$ is non-trivial. This implies that $u$ is a non-trivial non-negative weak solution to the Dirichlet problem 
\begin{align}\label{euler-lag-proof}
(-\Delta)^s_{\Omega}u=\mu_{N,s}(\Omega)~\frac{u}{\delta^{2s}}\text{ in }\Omega \quad \text{ and }\quad u=0 \text{ on } \partial\Omega.
\end{align}
By the strong maximum principle (see \cite[Proposition 1.7]{abatangelo2023hopf}), we obtain that $u >0$ almost everywhere in $\Omega$.
\par
Next, we show the uniqueness for the positive minimizer of $\mu_{N,s}(\Omega)$. For this, it is sufficient to use Picone's type inequality for regional fractional Laplacian.
Let us take $u, v \in H^s_0(\Omega)$ two positive minimizers of $\mu_{N,s}(\Omega)$. By Proposition 2.3 in \cite{frank2009sharp}, we have
\begin{align*}
[u]_s-\mu_{N,s}(\Omega)\int_\Omega\frac{u^2}{\delta^{2s}}\ dx\geq\frac{c_{N,s}}{4}\iint_{\Omega\times\Omega}\frac{(w(x)-w(y))^2}{|x-y|^{N+2s}}\,v(x)v(y) dxdy,
\end{align*}
where $w$ is such that $u=vw$. Since $u$ is a minimizer we have
\begin{align*}
\iint_{\Omega\times\Omega}\frac{(w(x)-w(y))^2}{|x-y|^{N+2s}}\,v(x)v(y) dxdy=0.
\end{align*}
Recalling that $v> 0$, we deduce that $w$ is constant. Hence $u= C v$ and then by the normalization assumption, \eqref{normalization*}, $C=1$. Thus $u=v$. This concludes the proof of the corollary.
\end{proof}

\section{Proof of Theorems \ref{Geom-convex-Hardy}~\&~\ref{second-main-result}}\label{section:proof of second main result}

In this section, we prove Theorems \ref{Geom-convex-Hardy}~\&~\ref{second-main-result}. The main ingredient is a new improved fractional Hardy inequality that we establish next.

\subsection{Geometric Hardy inequality}
In what follows, we will use the following notation. For $\nu\in\mathbb{S}^{N-1}$and $x\in\Omega$ we let: 
\begin{align*}
& \tau_{\nu}(x)=\inf\{t>0: x+t\nu\notin\Omega\},\quad d_{\nu,\Omega}(x)=\inf\{|t|: x+t\nu\notin\Omega \}, \\ \vspace{5mm}
& D_{\nu}(x)=\tau_\nu(x)+\tau_{-\nu}(x),\; \mathfrak{D}_\nu(x)=\sup\{|t|: x+t\nu\in\Omega \},\\ \vspace{5mm}
&\Omega_x:=\big\{z\in\Omega:\, x+t(z-x)\in\Omega,\, \forall t\in[0,1]\big\}.
\end{align*}
Observe that 
\begin{align*}
d_{\nu,\Omega}(x)=\min_{\Omega}\left\{\tau_{\nu}(x),\tau_{-\nu}(x)\right\},\quad
\delta(x)=\inf_{\mathbb{S}^{N-1}}\tau_\nu(x)\quad\text{and} \quad 
\mathfrak{D}_\nu(x)\leq D_{\nu}(x).\qquad\qquad
\end{align*}
We also have the following key observation.
\begin{align*}
|\Omega_x|:=\int_{\Omega_x}1\,dx&=\int_{\mathbb{S}^{N-1}}\int_0^{\tau_\nu(x)}t^{N-1}\,dtd\sigma(\nu)\\
&=\frac{1}{N}\int_{\mathbb{S}^{N-1}}\tau_\nu(x)^{N}\,d\sigma(\nu).
\end{align*}
Thus
\begin{align*}
|\Omega_x|=\frac{1}{N}\int_{\mathbb{S}^{N-1}}\tau_\nu(x)^{N}\,d\sigma(\nu).
\end{align*}
Let us now recall some notations from \cite{loss2010hardy}. For $0<\alpha<2$ we let
\begin{align*}
&M_{\alpha}(x)^{-\alpha}=\left(\int_{\mathbb{S}^{N-1}}|\cos(e,\nu)|^{\alpha}\,d\sigma(\nu)\right)^{-1}\displaystyle\int_{\mathbb{S}^{N-1}}\bigg(\frac{1}{d_{\nu,\Omega}(x)}+\frac{1}{\mathfrak{D}_\nu(x)}\bigg)^{\alpha}\,d\sigma(\nu),\\
\text{ and }&\nonumber\\
&m_{\alpha}(x)^{-\alpha}=\left(\int_{\mathbb{S}^{N-1}}|\cos(e,\nu)|^{\alpha}\,d\sigma(\nu)\right)^{-1}\displaystyle\int_{\mathbb{S}^{N-1}}\frac{1}{d_{\nu,\Omega}(x)^{\alpha}}\,d\sigma(\nu).
\end{align*}
In what follows we normalize $d\sigma(\nu)$ by letting $d\omega(\nu)=\displaystyle\frac{d\sigma(\nu)}{|\mathbb{S}^{N-1}|}$ so that $\displaystyle\int_{\mathbb{S}^{N-1}}d\omega(\nu)=1$, then we have 
\begin{align}
\nonumber&M_{\alpha}(x)^{-\alpha}=\frac{\sqrt{\pi}\,\Gamma(\frac{N+\alpha}{2})}{\Gamma(\frac{N}{2})\Gamma(\frac{1+\alpha}{2})}\displaystyle\int_{\mathbb{S}^{N-1}}\bigg(\frac{1}{d_{\nu,\Omega}(x)}+\frac{1}{\mathfrak{D}_\nu(x)}\bigg)^{\alpha}\,d\omega(\nu),\\
\text{ and }&\nonumber\\
\nonumber&m_{\alpha}(x)^{-\alpha}=\frac{\sqrt{\pi}\,\Gamma(\frac{N+\alpha}{2})}{\Gamma(\frac{N}{2})\Gamma(\frac{1+\alpha}{2})}\displaystyle\int_{\mathbb{S}^{N-1}}\frac{1}{d_{\nu,\Omega}(x)^{\alpha}}\,d\omega(\nu).
\end{align}
Here we use the fact that for any $e\in\mathbb{S}^{N-1}$
\begin{align}\label{cos-intg}
\int_{\mathbb{S}^{N-1}}|\cos(e,\nu)|^{\alpha}\,d\omega(\nu)=\frac{\Gamma(\frac{N}{2})\Gamma(\frac{1+\alpha}{2})}{\sqrt{\pi}\,\Gamma(\frac{N+\alpha}{2})}.
\end{align}
Furthermore, it holds that
\begin{align}\label{omega_x}
|\Omega_x|=\frac{|\mathbb{S}^{N-1|}}{N}\int_{\mathbb{S}^{N-1}}\tau_\nu(x)^{N}\,d\omega(\nu).
\end{align}
Finally, for a given convex open set $\Omega$ a classical geometric argument shows that 
\begin{align*}
m_{\alpha}(x)\leq \delta(x),\quad \text{ for all } x\in\Omega.
\end{align*}
Indeed, for $x\in\Omega$ let $e\in\mathbb{S}^{N-1}$ be such that $d_{e,\Omega}(x)=\delta(x)$ then 
\begin{align*}
d_{\nu,\Omega}(x)\cos(e,\nu)\leq\delta(x),\quad \text{ for all }\, \nu\in\mathbb{S}^{N-1}.
\end{align*}
Hence, 
\begin{align}\label{m_d_geom}
\int_{\mathbb{S}^{N-1}}\frac{1}{d_{\nu,\Omega}(x)^{\alpha}}\,d\omega(\nu)&\geq \int_{\mathbb{S}^{N-1}}|\cos(e,\nu)|^{\alpha}\frac{1}{\delta(x)^{\alpha}}\,d\omega(\nu) \nonumber\\
&\geq \int_{\mathbb{S}^{N-1}}|\cos(e,\nu)|^{\alpha}\,d\omega(\nu)\,\frac{1}{\delta(x)^{\alpha}}\\
&=\frac{\Gamma(\frac{N}{2})\Gamma(\frac{1+\alpha}{2})}{\sqrt{\pi}\,\Gamma(\frac{N+\alpha}{2})}\frac{1}{\delta(x)^{\alpha}},\nonumber
\end{align}
where we have used \eqref{cos-intg} in the last equality.

Next, we show that the following geometric fractional Hardy-type inequality holds.
\begin{thm}\label{Geom-Hardy-2} 
Let $\Omega\subset\R^N$ be an open set. Then
\begin{equation}\label{geo}
\frac{c_{N,s}}{2}\iint_{\Omega\times\Omega}\frac{(u(x)-u(y))^2}{|x-y|^{N+2s}}\ dxdy\geq\mathfrak{h}_{N,s}\int_{\Omega}\frac{u^2}{m_{2s}^{2s}}\ dx+a(N,s)\int_{\Omega}\frac{u^2}{|\Omega_x|^{\frac{2s}{N}}}\ dx,
\end{equation}
for all $u\in H^s_0(\Omega)$, where $a(N,s)$ is the constant given in \eqref{geometrichardyconstant0}.
\end{thm}


Before proving Theorem \ref{Geom-Hardy-2}, we need some auxiliary lemmas.
\begin{lemma}\label{alg-inq}
Let $a,b,\theta$ be real numbers such that $a\geq b>0$ and $1<\theta\leq2$. Then 
\begin{align}\label{alg-inq*}
(a+b)^\theta\geq a^\theta+(2^\theta-1) b^\theta.
\end{align}
\end{lemma}
\begin{proof}
Let $t=\frac{a}{b}$. Then \eqref{alg-inq*} reduces to
\begin{align*}
(t+1)^\theta\geq t^\theta+2^{\theta}-1.
\end{align*}
Let $f(t):=(t+1)^\theta-t^\theta$, then $f'(t):=\theta(t+1)^{\theta-1}-\theta t^{\theta-1}\geq 0$. Thus $f$ is non-decreasing and
\begin{align*}
f(t)\geq f(1)=2^\theta-1.
\end{align*}
and the lemma follows.    
\end{proof}
\begin{lemma}\label{D-lower-estimate} We have
\begin{equation*}
\displaystyle\int_{\mathbb{S}^{N-1}}\bigg(\frac{1}{\mathfrak{D}_\nu(x)}\bigg)^{2s}\,d\omega(\nu)\geq2^{-2s}\left(\frac{N|\Omega_x|}{|\mathbb{S}^{N-1}|}\right)^{-\frac{2s}{N}}.
\end{equation*}
\end{lemma}

\begin{proof}
We follow the argument used in \cite{Tidblom}. First, observe that, since $\mathfrak{D}_\nu(x)\leq D_\nu(x)$, we have
$$\displaystyle\int_{\mathbb{S}^{N-1}}\bigg(\frac{1}{\mathfrak{D}_\nu(x)}\bigg)^{2s}\,d\omega(\nu)\geq \displaystyle\int_{\mathbb{S}^{N-1}}\bigg(\frac{1}{D_\nu(x)}\bigg)^{2s}\,d\omega(\nu). $$
Jensen's inequality yields
$$
\displaystyle\int_{\mathbb{S}^{N-1}}\left(\frac{1}{D_{\nu}(x)}\right)^{2s} d \omega(\nu) \geq\left(\int_{\mathbb{S}^{N-1}}D_{\nu}(x) d \omega(\nu)\right)^{-2s}.
$$
The right-hand side could be estimated as follows. We have
$$
\begin{aligned}
\displaystyle\int_{\mathbb{S}^{N-1}}D_{\nu}(x) d \omega(\nu) & =\int_{\mathbb{S}^{N-1}} \left(\tau_{\nu}(x)+\tau_{-\nu}(x)\right) \,d \omega(\nu) \\
& =2\int_{\mathbb{S}^{N-1}} \tau_{\nu}(x) \,d \omega(\nu) \\
& \leq2\left(\int_{\mathbb{S}^{N-1}} \tau_{\nu}^{N}(x) d \omega(\nu)\right)^{\frac{1}{N}} \\
& =2\left(\frac{N\left|\Omega_{x}\right|}{\left|\mathbb{S}^{N-1}\right|}\right)^{\frac{1}{N}},
\end{aligned}
$$
where we have used \eqref{omega_x} in the last inequality. This concludes the proof.
\end{proof}
We now give the proof of Theorem \ref{Geom-Hardy-2}.

\begin{proof}[Proof of Theorem \ref{Geom-Hardy-2}.]
First, we recall the following fractional Hardy inequality from \cite{loss2010hardy}
\begin{equation}\label{loss-Hardy}
\frac{c_{N,s}}{2}\iint_{\Omega\times\Omega}\frac{(u(x)-u(y))^2}{|x-y|^{N+2s}}\ dxdy\geq\mathfrak{h}_{N,s}\int_{\Omega}\frac{u^2}{M_{2s}(x)^{2s}}\ dx.
\end{equation}
Now observe that, by using Lemma \ref{alg-inq} with $a=\displaystyle\frac{1}{d_{\nu,\Omega}(x)}$, $b=\displaystyle\frac{1}{\mathfrak{D}_\nu(x)}$ and $\theta=2s$, we have
\begin{align*}
M_{2s}(x)^{-2s}\geq m_{2s}(x)^{-2s}+(2^{2s}-1)\frac{\sqrt{\pi}\,\Gamma(\frac{N+2s}{2})}{\Gamma(\frac{N}{2})\Gamma(\frac{1+2s}{2})}\displaystyle\int_{\mathbb{S}^{N-1}}\bigg(\frac{1}{\mathfrak{D}_\nu(x)}\bigg)^{2s}\,d\omega(\nu).
\end{align*}
In view of Lemma \ref{D-lower-estimate}, it holds
\begin{align*}
M_{2s}(x)^{-2s}\geq m_{2s}(x)^{-2s}+2^{-2s}(2^{2s}-1)\frac{\sqrt{\pi}\,\Gamma(\frac{N+2s}{2})}{\Gamma(\frac{N}{2})\Gamma(\frac{1+2s}{2})}\left(\frac{N}{|\mathbb{S}^{N-1}|}\right)^{-\frac{2s}{N}}|\Omega_x|^{-\frac{2s}{N}}.
\end{align*}
Inserting this into \eqref{loss-Hardy}, inequality \eqref{geo} follows with
$$
a(N,s)=\mathfrak{h}_{N,s}2^{-2s}(2^{2s}-1)\frac{\sqrt{\pi}\,\Gamma(\frac{N+2s}{2})}{\Gamma(\frac{N}{2})\Gamma(\frac{1+2s}{2})}\left(\frac{N}{|\mathbb{S}^{N-1}|}\right)^{-\frac{2s}{N}}.
$$
This completes the proof of Theorem \ref{Geom-Hardy-2} and with it the proof of Theorem \ref{Geom-Hardy}.
\end{proof}

\begin{proof}[Proof of Theorem \ref{Geom-convex-Hardy}.]
Since $\Omega$ is convex then $m_{2s}(x)^{-2s}\geq \delta(x)^{-2s}$ and $\Omega_x=\Omega$. Taking this into account in Theorem \ref{Geom-Hardy}, the proof follows.
\end{proof}

\begin{proof}[Proof of Theorem \ref{Geom-convex-Hardy-bndd}.]
First, we recall the following fractional Hardy inequality from \cite{dyda remder}
\begin{align}\label{dyda-Hardy}
\frac{c_{N,s}}{2}\iint_{\Omega\times\Omega}&\frac{(u(x)-u(y))^2}{|x-y|^{N+2s}}\ dxdy\geq\mathfrak{h}_{N,s}\int_{\Omega}\frac{u^2}{M_{2s}(x)^{2s}}\ dx\nonumber\\
&+c_{N,s}~\frac{\pi^{\frac{N-1}{2}}\Gamma(s)(4-2^{3-2s})}
{2s\Gamma(\frac{N+2s-1}{2})}\frac{1}{\diam(\Omega)}\int_{\Omega}\frac{u^2}{M_{2s-1}(x)^{2s-1}}\ dx.
\end{align}
Arguing as in the proof of Theorem \ref{Geom-Hardy} and using that $\Omega$ is convex, we have
\begin{align}\label{m2s}
M_{2s}(x)^{-2s}\geq \delta(x)^{-2s}+2^{-2s}(2^{2s}-1)\frac{\sqrt{\pi}\,\Gamma(\frac{N+2s}{2})}{\Gamma(\frac{N}{2})\Gamma(\frac{1+2s}{2})}\left(\frac{N}{|\mathbb{S}^{N-1}|}\right)^{-\frac{2s}{N}}|\Omega|^{-\frac{2s}{N}}.
\end{align}
On the other hand, 
\begin{align}\label{m2s-1}
\nonumber M_{2s-1}(x)^{-2s+1}&\geq\frac{\sqrt{\pi}\,\Gamma(\frac{N+2s-1}{2})}{\Gamma(\frac{N}{2})\Gamma(s)}\displaystyle\int_{\mathbb{S}^{N-1}}\bigg(\frac{2}{\mathfrak{D}_\nu(x)}\bigg)^{2s-1}\,d\omega(\nu)\\
&\geq\frac{\sqrt{\pi}\,\Gamma(\frac{N+2s-1}{2})}{\Gamma(\frac{N}{2})\Gamma(s)}\left(\frac{N}{|\mathbb{S}^{N-1}|}\right)^{-\frac{2s-1}{N}}|\Omega|^{-\frac{2s-1}{N}},
\end{align}
where in the latter, we have used Lemma \ref{D-lower-estimate} with $2s$ replaced by $2s-1$, and the convexity of $\Omega$. Inserting \eqref{m2s} and \eqref{m2s-1} into \eqref{dyda-Hardy} the theorem follows with 
\begin{align*}
\mathfrak{C}(\Omega,N,s)=\left[a(N,s)+b(N,s)~\frac{|\Omega|^{\frac 1N}}{\diam(\Omega)}\right] |\Omega|^{-\frac{2s}{N}},\qquad
\end{align*}
and 
\begin{align*}
b(N,s)=2^{2s-1}(4-2^{3-2s})\frac{\Gamma(\frac{N+2s}{2})}
{\Gamma(\frac{N}{2})\Gamma(1-s)}\left(\frac{N}{|\mathbb{S}^{N-1}|}\right)^{-\frac{2s-1}{N}}. 
\end{align*}
\end{proof}

\begin{proof}[Proof of Theorem \ref{Geom-convex-Hardy-bndd+}.]
Arguing as in the proof of Theorem \ref{Geom-convex-Hardy-bndd} and using that 
\begin{align*}
M_{2s-1}(x)^{-2s+1}\geq m_{2s-1}(x)^{-2s+1}\geq \delta(x)^{-2s+1},
\end{align*}
it follows from \eqref{dyda-Hardy} that \begin{equation*}\label{}
\frac{c_{N,s}}{2}\iint_{\Omega\times\Omega}\frac{(u(x)-u(y))^2}{|x-y|^{N+2s}}\ dxdy\geq\mathfrak{h}_{N,s}\int_{\Omega}\frac{u^2}{\delta^{2s}}\ dx+a(N,s)|\Omega|^{-\frac{2s}{N}}\int_{\Omega}u^2\ dx+\mathcal{K}\int_{\Omega}\frac{u^2}{\delta^{2s-1}}\ dx,
\end{equation*}
for all $u\in H^s_0(\Omega)$, with
\begin{equation*}
\mathcal{K}=\mathcal{K}(\Omega,N,s)=2^{2s-1}(4-2^{3-2s})\frac{\Gamma(s)\Gamma(\frac{N+2s}{2})}
{\Gamma(1-s)\Gamma(\frac{N+2s-1}{2})}\diam(\Omega)^{-1}.
\end{equation*}
This concludes the proof of the Theorem.
\end{proof}

\begin{proof}[Proof of Theorem \ref{second-main-result}] The proof is a consequence of Theorem \ref{Geom-convex-Hardy} and \eqref{a9}.
\end{proof}
\begin{remark}
We notice that Theorem \ref{second-main-result} could be proved by using the following improved Hardy inequality by Dyda and Frank, see \cite{dyda2012fractional}. For any open set $\Omega\subset \mathbb{R}^N$ there exists $\mathcal{C}(N,s)>0$ such that $\forall u\in H^s_0(\Omega)$
\begin{equation}\label{improved dyda-frank}
\frac{c_{N,s}}{2}\displaystyle\iint_{\Omega\times\Omega}\frac{(u(x)-u(y))^2}{|x-y|^{N+2s}}\ dxdy-\mathfrak{h}_{N,s}\int_{\Omega}\frac{u(x)^2}{m_{2s}(x)^{2s}}\,dx\geq \mathcal{C}(N,s)\left(\int_{\Omega}|u|^{\frac{2N}{N-2s}}\right)^{\frac{N-2s}{N}}.
\end{equation}
Since $\Omega$ is convex, then $m_{2s}(x)\leq \delta(x)$. Using Holder's inequality we have $$\left(\int_{\Omega}|u|^{\frac{2N}{N-2s}}\right)^{\frac{N-2s}{N}}\geq|\Omega|^{-\frac{2s}{N}}\int_\Omega u^2\,dx.$$
Coming back to \eqref{improved dyda-frank}, we get
\begin{equation*}
\frac{c_{N,s}}{2}\iint_{\Omega\times\Omega}\frac{(u(x)-u(y))^2}{|x-y|^{N+2s}}\ dxdy-\mathfrak{h}_{N,s}\int_{\Omega}\frac{u(x)^2}{\delta(x)^{2s}}\,dx\geq \mathcal{C}(N,s)|\Omega|^{-\frac{2s}{N}}\int_{\Omega}u^2\ dx,
\end{equation*}
that is
\begin{equation}\label{f}
\frac{\frac{c_{N,s}}{2}\displaystyle\iint_{\Omega\times\Omega}\frac{(u(x)-u(y))^2}{|x-y|^{N+2s}}\ dxdy-\mathfrak{h}_{N,s}\int_{\Omega}\frac{u(x)^2}{\delta(x)^{2s}}\,dx}{\displaystyle\int_{\Omega}u^2\ dx}\geq \mathcal{C}(N,s)|\Omega|^{-\frac{2s}{N}}.
\end{equation}
The result, therefore, follows by taking the infimum in \eqref{f}. Nevertheless, this argument does not yield an explicit expression for the constant $\mathcal{C}(N,s)$.
 \end{remark}

\section{Proof of Theorems \ref{thm_lmada_asym} \& \ref{thm_hardy_asym}}\label{main-asym}
This section is devoted to the prove of Theorems \ref{thm_lmada_asym} \& \ref{thm_hardy_asym}. First, we recall the following Sobolev embeddings, see \cite{di2012}.
\begin{prop}
Let $\Omega$ be an open set in $\R^N$, $0 < s<s' < 1$ and $u : \Omega\to \R$
be a measurable function. Then
$$ \|u\|^2_{H^s(\Omega)}\leq C~\|u\|^2_{H^{s'}(\Omega)},$$
for some suitable positive constant $C = C(N,s)\geq 1$. In particular,
$$H^{s'}(\Omega)\subset H^{s}(\Omega).$$
\end{prop}
The space $H^{\frac{1}{2}}(\Omega)$ is of particular interest, since it contains constant functions and the Hardy constant $\mu_{N,\frac{1}{2}}(\Omega)=0$ for any open set. These facts will play a crucial role in understanding the asymptotic behavior of the Hardy constant for $s$ close to $\frac{1}{2}$.

Next, we aim to show continuity of the Gagliardo semi-norm with respect to $s$. We recall the elementary inequality 
\begin{equation}\label{elementary-identity-1}
|e^t-1|\leq\sum_{k=1}^{+\infty}\frac{|t|^k}{k!}\leq\sum_{k=1}^{+\infty}\frac{|t|^k}{(k-1)!}\leq|t|e^{|t|},~~~\textnormal{for}~t\in\R.
\end{equation}
Given $r,\gamma>0$, we have the following growth for a logarithmic function:
\begin{equation}\label{logarithmic-decays}
|\log|z||\leq\frac{1}{e\gamma}|z|^{-\gamma}~~\text{if}~~ |z|\leq r~~~~\text{and}~~~~|\log|z||\leq\frac{1}{e\gamma}|z|^{\gamma}~~ \text{if}~~|z|\geq r.
\end{equation}
\begin{lemma}\label{lambda_conti}
Let $\phi\in H^s_0(\Omega)$. Then the map 
\begin{align*}
\left[\frac{1}{2},\,1\right)\to (0,+\infty),\quad s\mapsto f(s):=[\phi]_{s}
\end{align*}
is continuous in $s$.
\end{lemma}
\begin{proof} By density we consider only the case $\phi\in C^\infty_c(\Omega)$. Fix $s_0\in\left[\frac{1}{2},\,1\right)$.
Then
\begin{align*}
\Big|[\phi]_s- [\phi]_{s_0}\Big|
&\leq\frac{1}{2}|c_{N,s}-c_{N,s_0}|[\phi]_{s_0}+\frac{c_{N,s}}{2}\iint_{\Omega\times\Omega}\frac{(\phi(x)-\phi(y))^2}{|x-y|^{N+2s_0}}||x-y|^{2(s_0-s)}-1|\ dxdy\\
&\leq\frac{1}{2}|c_{N,s}-c_{N,s_0}|[\phi]_{s_0}+\frac{c_{N,s}}{2}\iint_{\Omega\times\Omega}\frac{(\phi(x)-\phi(y))^2}{|x-y|^{N+2s_0}}||x-y|^{2(s_0-s)}-1|\ dxdy.
\end{align*}
In view of \eqref{elementary-identity-1} we have that
\begin{align*}
||x-y|^{2(s_0-s)}-1|=|e^{2(s_0-s)\log|x-y|}-1|&\leq2|s_0-s||\log|x-y||e^{2|s_0-s||\log|x-y||}\\&\leq 2|s_0-s||\log|x-y|||x-y|^{2|s_0-s|}.
\end{align*}
Hence, 

\begin{align}\label{phi-phi estim}
\nonumber&\Big|[\phi]_s- [\phi]_{s_0}\Big|\\
&\leq\frac{1}{2}|c_{N,s}-c_{N,s_0}|[\phi]_{s_0}+c_{N,s}|s-s_0|\iint_{\Omega\times\Omega}\frac{(\phi(x)-\phi(y))^2}{|x-y|^{N+2s_0}}||\log|x-y|||x-y|^{2|s-s_0|}\ dxdy\nonumber\\
&\leq \frac{1}{2}|c_{N,s}-c_{N,s_0}|[\phi]_{s_0}+c_{N,s}\diam(\Omega)^{2|s-s_0}|s-s_0|\iint_{\Omega\times\Omega}\frac{(\phi(x)-\phi(y))^2}{|x-y|^{N+2s_0}}|\log|x-y||dxdy,
\end{align}
where $\diam(\Omega)=\sup\{|x-y|:x,y\in\Omega\}$ is the diameter of $\Omega$.

Let us estimate the second term in the last inequality. Denote
$$ \mathbb{A}:=\iint_{\Omega\times\Omega}\frac{(\phi(x)-\phi(y))^2}{|x-y|^{N+2s_0}}|\log|x-y||dxdy.$$
Using \eqref{logarithmic-decays} with $\gamma <2(1-s_0)$, we find that
\begin{align*}
\mathbb{A}&\leq\iint_{\Omega\times\Omega}\frac{(\phi(x)-\phi(y))^2}{|x-y|^{N+2s_0}}|\log|x-y||\ dxdy\\
&=\iint_{\{|x-y|\leq1\}}\frac{(\phi(x)-\phi(y))^2}{|x-y|^{N+2s_0}}|\log|x-y||\ dxdy+\iint_{\{|x-y|>1\}}\frac{(\phi(x)-\phi(y))^2}{|x-y|^{N+2s_0}}|\log|x-y||\ dxdy\\
&\leq e\gamma\Bigg[\iint_{\{|x-y|\leq1\}}\frac{(\phi(x)-\phi(y))^2}{|x-y|^{N+2s_0+\gamma}}\ dxdy+\diam(\Omega)^\gamma\iint_{\{|x-y|>1\}}(\phi(x)-\phi(y))^2\, dxdy\Bigg]\\
&\leq e\gamma\Bigg[\|\phi\|_{C^1}\iint_{\{|x-y|\leq1\}}|x-y|^{2-N-2s_0-\gamma}\ dxdy+\diam(\Omega)^\gamma2\|\phi\|_{L^2(\Omega)}^2|\Omega|\Bigg]=:\mathbb{B}(\Omega,s_0,\gamma).
\end{align*}
Coming back to \eqref{phi-phi estim}, we obtain
\begin{align*}
\Big|[\phi]_s- [\phi]_{s_0}\Big|
&\leq \frac{1}{2}|c_{N,s}-c_{N,s_0}|[\phi]_{s_0}+c_{N,s}\diam(\Omega)^{2|s-s_0|}|s-s_0|\mathbb{B}(\Omega,s_0,\gamma).
\end{align*}
Since the map $s\mapsto c_{N,s}$ is $C^1$ in $(0,1)$ there exists $\bar{c}(N,s_0)>0$ such that, for $s\in (0,1)$, 
\begin{align*}
|c_{N,s}-c_{N,s_0}|\leq \bar{c}(N,s_0) |s-s_0|.    
\end{align*}
Hence, it follows that
\begin{align*}
\Big|[\phi]_s- [\phi]_{s_0}\Big|
\leq \Big[\frac{1}{2}\bar{c}(N,s_0)[\phi]_{s_0}+c_{N,s}\diam(\Omega)^{2|s-s_0|}\mathbb{B}(\Omega,s_0,\gamma)\Big]|s-s_0|.
\end{align*}
From this, the lemma easily follows.
\end{proof}
\begin{lemma}\label{w-lower esti}
Let $\Omega$ be a bounded $C^{1,1}$ open set such that $\mu_{N,s}(\Omega)<\mathfrak{h}_{N,s}$. Let $w$ be a solution to
\begin{align*}
(-\Delta)^s_{\Omega}w=\mu_{N,s}(\Omega)~\frac{w}{\delta^{2s}}\text{ in }\Omega,\quad w>0 \text{ in }\Omega, \quad \text{ and }\quad w=0 \text{ on } \partial\Omega.
\end{align*}
Then there exists $c=c(N,s)>0$ such that, for all $x\in\Omega$,
\begin{align*}
w(x)\geq c\delta(x)^{2s-1}.
\end{align*}
\end{lemma}
\begin{proof}
By regularity results, we know that $w\in C(\Omega)$ and by the Green representation formula we have
\begin{align}\label{w-rep-green}
w(x)=\mu_{N,s}(\Omega)\int_\Omega G(x,y)~\frac{w(y)}{\delta(y)^{2s}}\,dy,
\end{align}
where $G(\cdot,\cdot)$ denotes the Green function for regional fractional Laplacian. By the estimates in \cite{chen-green} we have 
\begin{align*}
G(x,y)&\geq c \min\left( \frac{1}{|x-y|^{N-2s}},\frac{\delta(x)^{2s-1}\delta(y)^{2s-1}}{|x-y|^{N-2+2s}}\right)\\
&\geq c \min(\diam(\Omega)^{N-2s},\diam(\Omega)^{N-2+2s})\min\left(1,\delta(x)^{2s-1}\delta(y)^{2s-1}\right).
\end{align*}
Using this in \eqref{w-rep-green}, we obtain
\begin{align*}\label{}
w(x)\geq \bar c ~\mu_{N,s}(\Omega)\int_\Omega \min\left(1,\delta(x)^{2s-1}\delta(y)^{2s-1}\right)~\frac{w(y)}{\delta(y)^{2s}}\,dy,
\end{align*}
where $\bar c=\bar c(\Omega, N, s)$ is a positive constant. Hence
\begin{align*}\label{}
\frac{w(x)}{\delta(x)^{2s-1}}\geq \bar c ~\mu_{N,s}(\Omega)\int_\Omega \min\left(\delta(x)^{1-2s},\delta(y)^{2s-1}\right)~\frac{w(y)}{\delta(y)^{2s}}\,dy.
\end{align*}
Since the right-hand side can be bounded from below by a positive constant, the Lemma follows. 
\end{proof}
\subsection{Proof of Theorem \ref{thm_hardy_asym}}
In what follows, we give the proof of Theorem \ref{thm_hardy_asym}.
\begin{proof}
By Contradiction, assume that there exists $s_n$ such that $\frac{1}{2}<s_n$, $s_n\to\frac{1}{2}$ and 
\begin{align}\label{u_n vs h_n}
\mu_n(\Omega):=\mu_{N,s_n}(\Omega)<\mathfrak{h}_{N,s_n}, ~~~~\text{ for all } n\in\mathbb{N}.
\end{align}
Then, by Theorem \ref{first-main-result}, $\mu_n(\Omega)$ is achieved and there exists $u_n\in H_0^{s_n}(\Omega)$ such that
\begin{align}\label{u_n equn}
(-\Delta)^{s_n}_{\Omega}u_n=\mu_n(\Omega)~\displaystyle\frac{u_n}{\delta^{2s_n}}~~\textnormal{in}~\Omega \quad \text{ and }\quad u_n=0 \text{ on } \partial\Omega,
\end{align}
normalized by $\displaystyle\int_{\Omega}\frac{u^2_n}{\delta^{2s_n}}=1$. Using $u_n$ as a test function in \eqref{u_n equn} we have
\begin{align}\label{u_nto0}
\left[u_n\right]_{s_n}=\mu_n(\Omega).
\end{align}
By Lemma \ref{lambda_conti}, \eqref{u_n vs h_n}, \eqref{u_nto0} and using that $\mathfrak{h}_{N,s_n}\to0$ as $n\to\infty$, it holds that 
\begin{align*}
\lim_{n\to+\infty}\left[u_n\right]_{s_n}=0.
\end{align*}
Hence, for $n$ large enough, we have
\begin{align}\label{W2s<1}
\iint_{\Omega\times\Omega}\frac{(u_n(x)-u_n(y))^2}{|x-y|^{N+2s_n}}\ dxdy<1.
\end{align}
 On the other hand, Theorem \ref{valdi} with 


\begin{equation*}
\frac{1}{q}<\tau<\frac{1}{2} \quad\text{ and }\quad 2<q<\frac{2N}{N-2(s_n-\tau)},
\end{equation*}
and \eqref{W2s<1} give
\begin{align*}
1>\iint_{\Omega\times\Omega}\frac{(u_n(x)-u_n(y))^2}{|x-y|^{N+2s_n}}\ dxdy\geq C(N,s_n,q,\tau,\Omega)\iint_{\Omega\times\Omega}\frac{|u_n(x)-u_n(y)|^{q}}{|x-y|^{N+q\tau}}\ dxdy.
\end{align*}
By the fractional Hardy inequality, see for instance \cite{dyda2004fractional,loss2010hardy, DV}, the above inequality becomes 
\begin{align}\label{2-p inq}
1> \mu_{N,q,\tau}(\Omega)~C(N,s_n,q,\tau,\Omega)\displaystyle\int_{\Omega}\frac{|u_n(x)|^{q}}{\delta(x)^{q\tau}}~dx.
\end{align}
Notice that $\mu_{N,q,\tau}(\Omega)>0$ since $q\tau>1$.
Recalling Lemma \ref{w-lower esti} we have
\begin{align*}
	u_n(x)\geq c\delta(x)^{2s_n-1}, \text{ for some constant } c=c(N,s_n)>0,
\end{align*}
and therefore \eqref{2-p inq} yields
\begin{align*}
	1> c~\mu_{N,q,\tau}(\Omega)~C(N,s_n,q,\tau,\Omega)~\displaystyle\int_{\Omega}\delta(x)^{(2s_n-1)q-q\tau}\, dx=+\infty,
\end{align*}
where, in the latter, we have used the fact that
\begin{align*}
	(2s_n-1)q-q\tau+1<0,
\end{align*}
for $n$ large enough and by our choice of $q$ and $\tau$. This gives a contradiction and ends the proof of the first part of Theorem \ref{thm_hardy_asym}.\\


The second part of Theorem \ref{thm_hardy_asym} is a consequence of Theorem \ref{first-main-result}. The proof is complete. 
\end{proof}
\subsection{Proof of Theorem \ref{thm_lmada_asym}}

We prove, in what follows, the asymptotic behavior of $\lambda^*(s,\Omega)$ for $s$ close to $\frac{1}{2}$ for any bounded $C^{1,1}$ open set $\Omega$. 
\begin{proof} First, from the definition of $\lambda^*(s,\Omega)$ one could write
\begin{equation}\label{lambda*_def}
\lambda^*(s,\Omega)=\inf\left\{[\phi]_s-\mathfrak{h}_{N,s}\int_{\Omega}\frac{\phi^2}{\delta^{2s}}\, dx:\, \phi\in H^s_0(\Omega) \text{ and } \int_{\Omega}\phi^2\ dx=1\right\}. 
\end{equation}
We observe that by density we may assume that $\phi\in C_c^\infty(\Omega)$ , then by Lemma \ref{lambda_conti} we have 
\begin{align*}
\lim_{s\to\frac{1}{2}}[\phi]_s=[\phi]_{\frac{1}{2}}.
\end{align*}
Furthermore, by direct computation, $\mathfrak{h}_{N,s}\to 0$ as $s\to\frac{1}{2}$. 
In addition, 
\begin{align*}
\int_{\Omega}\frac{\phi^2}{\delta^{2s}}(x)\,dx&=\int_{\Omega}\frac{\phi^2}{\delta^{2}}(x)\delta^{2-2s}(x)\,dx\\
&\leq \diam(\Omega)^{2-2s}\int_{\Omega}\frac{\phi^2}{\delta^{2}}(x)\,dx\\
&\leq \diam(\Omega)^{2-2s} \mu(\Omega)^{-1}\int_{\Omega}|\nabla\phi|^2\,dx,
\end{align*}
where we have used the classical Hardy inequality in the last inequality, and $\mu(\Omega)$\footnote{ Recall that by \cite{Necas} we know that for a bounded Lipschitz domain $\mu(\Omega)>0$.} denotes the classical Hardy constant. From \eqref{lambda*_def} we have
\begin{align*}
\lambda^*(s,\Omega)\leq [\phi]_s-\mathfrak{h}_{N,s}\int_{\Omega}\frac{\phi^2}{\delta^{2s}}\ dx,\quad\forall\phi\in C^{\infty}_c(\Omega)~~\textnormal{with}~~\int_{\Omega}\phi^2\ dx=1,
\end{align*}
from which we get, 
\begin{align*}
\limsup_{s\to\frac{1}{2}^+}\lambda^*(s,\Omega)\leq \lim_{s\to\frac{1}{2}^+}[\phi]_s=[\phi]_{\frac{1}{2}}~~~~\forall\phi\in C^{\infty}_c(\Omega)~~\textnormal{with}~~\int_{\Omega}\phi^2\ dx=1.
\end{align*}
 In particular, for $\phi=\displaystyle\frac{u_k}{\|u_k\|_{L^2(\Omega)}}$ where $u_k\in C^{0,1}_c(\Omega)$ is the sequence defined in \cite[Eq. (A.21)]{fall2023existence}, it follows that 
\begin{align*}
\limsup_{s\to\frac{1}{2}^+}\lambda^*(s,\Omega)\leq [v_k]_{\frac{1}{2}}=\frac{1}{\|u_k\|^2_{L^2(\Omega)}}[u_k]_{\frac{1}{2}}~~~\forall k.
\end{align*}
Thus passing to the limit as $k\to\infty$ and from \cite[Lemma A.5]{fall2023existence}, we obtain 
\begin{align}\label{limsup-lambda-star-2}
\limsup_{s\to\frac{1}{2}^+}\lambda^*(s,\Omega)\leq \lim_{k\to\infty}\frac{1}{\|u_k\|^2_{L^2(\Omega)}}[u_k]_{\frac{1}{2}}=\frac{1}{|\Omega|}[1_{\Omega}]_{\frac{1}{2}}=0. 
\end{align}
Finally, by Theorem \ref{thm_hardy_asym} there exists $s^*=s^*(\Omega)\in (\frac{1}{2},1)$ such that $\mu_{N,s}(\Omega)=\mathfrak{h}_{N,s}$ for all $s\in (\frac{1}{2},s^*)$ and then
$$\lambda^*(s,\Omega)\geq0, \quad \text{ for all } s\in (\frac{1}{2},s^*).$$
Thus by \eqref{limsup-lambda-star-2}, we infer that 
\begin{align*}
0\leq\liminf_{s\to\frac{1}{2}^+}\lambda^*(s,\Omega)\leq \lim_{s\to\frac{1}{2}^+}\lambda^*(s,\Omega)\leq \limsup_{s\to\frac{1}{2}^+}\lambda^*(s,\Omega)\leq0.
\end{align*}
This ends the proof of the Theorem.
\end{proof}

\appendix 

\section{An asymptotic estimate.} \label{appendix2}
In this appendix, we prove the following asymptotic result for
\begin{equation*}
    K(x)=\int_{\Omega}\frac{(\delta(y)-\delta(x))^2}{|y-x|^{N+2s}}\ dy,~~x\in\Omega_{\beta},
\end{equation*}
where we recall that $\Omega_{\beta}=\{x\in\Omega: \delta(x)<\beta\}$ with $\beta>0$ sufficiently small.
\begin{prop}\label{prop-appendix2}
Let $x\in\Omega_{\beta}$. Then
    \begin{equation*}
        K(x)\sim\delta(x)^{2-2s}~~~\textnormal{as}~~~\delta(x)\to0.
    \end{equation*}
\end{prop}

\begin{proof}
 By our assumption $\Omega$ is a $C^{1,1}$-domain. Therefore, locally at the boundary, the domain is a subgraph of a $C^{1,1}$ function. Thus for every $x_0\in\partial\Omega$, there exist $r>0$ and a $\phi\in C^{1,1}(\R^{N-1})$ function such that 
 \begin{equation*}
           \Omega\cap B_r(x_0)=\{x=(x',x_N): x_N>\phi(x')\}~~\textnormal{and}~~\partial\Omega\cap B_r(x_0)=\{(x',x_N): x_N=\phi(x')\}.
 \end{equation*}
 After translation and rotation, we may assume that $x_0=0$ so that
 \begin{equation*}
     \partial\Omega\cap B_r=\{(x',x_N): x_N=\phi(x')\},~~~\phi(0)=|\nabla\phi(0)|=0.
 \end{equation*}
 We consider the change of variables:
 \begin{equation*}
     X'=x',~~ X_N=x_N-\phi(x').
 \end{equation*}
 Then in the new variables, $\Omega\cap B_r\mapsto\{X_N>0\}$ and the boundary is $\{X_N=0\}$. Moreover, the boundary distance becomes $\delta(x)=x_N-\phi(x')+O(|x'|^2)$. So, very close to the boundary,
 \begin{equation*}
     \delta(x)\approx X_N.
 \end{equation*}
 Now, for all $x\in\Omega_{\beta}$ with $\beta$ sufficiently small,
 \begin{align*}
     K(x)&=\int_{\Omega\cap B_r}\frac{(\delta(x)-\delta(y))^2}{|x-y|^{N+2s}}\ dy+\int_{\Omega\setminus B_r}\frac{(\delta(x)-\delta(y))^2}{|x-y|^{N+2s}}\ dy\\
     &=:I(x)+J(x).
 \end{align*}
 Clearly,
 \begin{equation}\label{i1}
     J(x)=O(1).
 \end{equation}
 On the other hand, from the above discussion,
 \begin{align*}
     I(x)&\approx\int_{0}^{r}dY_N\int_{\R^{N-1}}\frac{(X_N-Y_N)^2}{(|X'-Y'|+|X_N-Y_N|^2)^{\frac{N+2s}{2}}}\ dY'\\
     &=\int_{0}^{r}(X_N-Y_N)^2dY_N\int_{\R^{N-1}}\frac{1}{(|X'-Y'|+|X_N-Y_N|^2)^{\frac{N+2s}{2}}}\ dY'\\
     &=\int_{0}^{r}(X_N-Y_N)^2dY_N\int_{\R^{N-1}}\frac{|X_N-Y_N|^{-N-2s}}{\Big(\Big|\frac{X'-Y'}{|X_N-Y_N|}\Big|^2+1\Big)^{\frac{N+2s}{2}}}\ dY'\\
     &=\int_{0}^{r}(X_N-Y_N)^2dY_N\int_{\R^{N-1}}\frac{|X_N-Y_N|^{-1-2s}}{(|Z|^2+1)^{\frac{N+2s}{2}}}\ dZ\\
     &=C(N,s)\int_{0}^{r}\frac{(X_N-Y_N)^2}{|X_N-Y_N|^{1+2s}}\ dY_N,
 \end{align*}
 where
 \begin{equation*}
     C(N,s)=\int_{\R^{N-1}}\frac{1}{(|Z|^2+1)^{\frac{N+2s}{2}}}\ dZ<\infty.
 \end{equation*}
 Set $t=X_N-Y_N$. If $Y_N<X_N$, then $t>0$. Moreover, if $Y_N>X_N$, $t<0$. So
 \begin{align}\label{i2}
   \nonumber  I(x)&\approx C(N,s)\left(\int_{0}^{X_N}\frac{t^2}{t^{1+2s}}\ dt+\int_{-r}^{0}\frac{t^2}{|t|^{1+2s}}\ dt\right)\\
   \nonumber  &=C(N,s)\left(\int_{0}^{X_N}t^{1-2s}\ dt+\int_{0}^{r}t^{1-2s}\ dt\right)\\
     &=C(N,s)\left(\frac{X_N^{2-2s}}{2-2s}+\frac{r^{2-2s}}{2-2s}\right).
 \end{align}
 From \eqref{i1} and \eqref{i2}, we deduce that
 \begin{equation*}
     K(x)\approx C\delta(x)^{2-2s}+O(1).
 \end{equation*}
 Thus,
 \begin{equation*}
     K(x)\sim\delta(x)^{2-2s}~~~\textnormal{as}~~~\delta(x)\to0.
 \end{equation*}
\end{proof}

\section{An embedding Theorem}
Let $s\in (0,1)$, $p\in (1, +\infty)$ and $\Omega$ be an open bounded set with Lipschitz boundary; we denote
\[
W^{s, p}(\Omega):=\left\{u\in L^p(\Omega) \mbox{ such that } [u]_{W^{s,p}(\Omega)}<+\infty \right\}
\]
with
\[
 [u]_{W^{s,p}(\Omega)}:=\left(\iint_{\Omega\times\Omega}\frac{|u(x)-u(y)|^p}{|x-y|^{N+sp}}\,dx\,dy \right)^{1/p},
\]
endowed with the norm
\begin{equation}\label{Wsp_norm}
\|u\|_{W^{s, p}(\Omega)}:=\left([u]^p_{W^{s, p}(\Omega)} + \|u\|^p_{L^p(\Omega)}\right)^{\frac1p} 
\end{equation}
\begin{thm}\label{valdi}
Let $s\in (0,1)$ and $p\in (1, +\infty)$ be such that~$1<sp<N$.
Let $\Omega\subset \R^N$ be a bounded domain with Lipschitz boundary. Assume that $\tau$ and $q$ satisfy,
\begin{equation*}
0\le \tau < s, \quad\&\quad 1 \le q\le\frac{Np}{N-(s-\tau)p}.
\end{equation*}
Then,  there exists a positive constant~$\mathcal{C}=\mathcal{C}(N,s,p,q,\tau,\Omega)$ such that, for any~$u\in W_0^{s,p}(\Omega)$,
	\vspace{3mm}
\begin{align}\label{valdi*}
\int_\Omega\int_\Omega\frac{|u(x)-u(y)|^{q}}{|x-y|^{N+q\tau}}\,dy\,dx\leq \mathcal{C} \int_\Omega\int_\Omega\frac{|u(x)-u(y)|^{p}}{|x-y|^{N+ps}}\,dy\,dx,
\end{align}
namely, the space $W_0^{s,p}(\Omega)$ is continuously embedded in $W_0^{\tau,q}(\Omega)$.
\end{thm}
\begin{proof}
Let $u\in W_0^{s,p}(\Omega)$, by \cite[Theorem 1.2]{valdi}, there exists a constant $C=C(N,s,p,q,\tau,\Omega)>0$ such that
\begin{equation}\label{wsp_wtq}
\|u\|_{W^{\tau,q}(\Omega)}
\le C\|u\|_{W^{s,p}(\Omega)}.
\end{equation}
Using that
\begin{equation*}
1<q\le\frac{Np}{N-(s-\tau)p}\leq\frac{Np}{N-sp},
\end{equation*}
\eqref{valdi*} follows from \eqref{Wsp_norm} and \eqref{wsp_wtq}.

\end{proof}

\textbf{Acknowledgements:} The second author is supported by The Fields Institute. This project was initiated at Rutgers University during a visit of the authors to the Department of Mathematics in October 2024. They wish to thank the Department for their kind hospitality.\\

\textbf{Data Availability}

No datasets were generated or analyzed during the current study.

\bibliographystyle{plain}

\end{document}